\documentclass[11pt]{amsart}
\usepackage[letterpaper,hmargin=1in,vmargin=1in]{geometry}

\usepackage{amsmath,amsthm,amssymb,graphics,color}
\usepackage{hyperref} 
\usepackage{xcolor}
\usepackage[english]{babel}
\usepackage{tikz}
\usepackage[]{geometry}
\usetikzlibrary{quotes,angles}
\usepackage{dsfont}

\usepackage{subcaption}
\usepackage{paralist}
\usepackage{graphicx}
\usepackage{esvect}
\usepackage{float}
\usepackage{array}
\usepackage{esvect}
\usepackage{float}
\usepackage{array}
\usepackage{tensor}
\usepackage{amsthm}
\usepackage{amsfonts}
\usepackage{bbm}

\newcommand{\C}{\mathbb{C}}

\newcommand{\Z}{\mathbb{Z}}
\newcommand{\N}{\mathbb{N}}

\newcommand{\R}{\mathbb{R}}

\newcommand{\LL}{\mathcal{L}}

\newcommand{\II}{I_{M, \rho}^{\sigma}}
\newcommand{\Dl}{\mathcal{D} \ell}
\newcommand{\Sl}{\mathcal{S} \ell}

\newcommand{\n}{\mathcal{N}}
\newcommand{\ka}{\kappa}

\newcommand{\Id}{\text{Id}}

\newcommand{\Tr}{\text{Tr}}

\newcommand{\Span}{\text{Span}}

\newcommand{\supp}{\text{supp}}

\renewcommand{\phi}{\varphi}
\renewcommand{\theta}{\vartheta}
\renewcommand{\epsilon}{\varepsilon}
\newcommand{\1}{\mathds{1}}

\newtheorem{theo}{Theorem}[section]
\newtheorem{prop}[theo]{Proposition}
\newtheorem{coro}[theo]{Corollary}
\newtheorem{lemm}[theo]{Lemma}

\theoremstyle{definition}
\newtheorem{def1}[theo]{Definition}
\theoremstyle{remark}
\newtheorem{rema}[theo]{Remark}

\newcommand{\nwc}{\newcommand}
\nwc{\Oph}{\operatorname{Op}_\hbar}
\nwc{\la}{\langle}
\nwc{\ra}{\rangle}

\nwc{\mf}{\mathbf} 
\nwc{\blds}{\boldsymbol} 
\nwc{\ml}{\mathcal} 

\renewcommand{\Im}{\operatorname{Im}}
\renewcommand{\Re}{\operatorname{Re}}

\newcommand{\eps}{\varepsilon}
\newcommand{\tr}{\text{Tr}\,}

\newcommand{\wt}{\widetilde}
\newcommand{\wh}{\widehat}

\newcommand{\sgn}{\text{sgn  \,}}

\renewcommand{\d}{\partial}

\renewcommand{\phi}{\varphi}

\newcommand{\lsp}{\text{LSP}}

\newcommand{\black}[1]{\color{black}}

\title{Balian-Bloch wave invariants for nearly degenerate orbits}
\date{}

\author{Vadim Kaloshin $^{1}$}
\author{Illya Koval $^{1}$}
\author{Amir Vig $^{2}$}
\address{$^1$ Institute of Science and Technology Austria, Klosterneuburg, Lower Austria}
\address{$^2$ Department of Mathematics, University of Michigan, Ann Arbor, MI 48109, USA}

\email{Vadim.Kaloshin@gmail.com}
\email{Illya.Koval@ist.ac.at}
\email{Vig@umich.edu}

\begin{document}

	\maketitle

\begin{abstract}
	This paper is part I of a series in which we aim to show that the singular support of the wave trace and the length spectrum of a smooth, strictly convex, and bounded planar billiard table are generally distinct objects. We derive an asymptotic trace formula for the regularized resolvent which is dual to the wave trace and contains the same information. To do this, we consider a class of periodic orbits which have nearly degenerate Poincar\'e maps and study their leading order behavior as the deformation parameter goes to zero, generating large coefficients in the wave trace. We also keep careful track of the Maslov indices, which will allow us to match contributions of opposite signs in our subsequent paper. Each cancellation of coefficients in the resolvent trace corresponds to making the wave trace one degree smoother. The resolvent based approach is due to Balian and Bloch and was significantly expanded upon by Zelditch in a foundational series of papers \cite{Zel09}, \cite{Zel0res}, \cite{Zelditch3} and \cite{Zelditch1}.
\end{abstract}

	\section{Introduction} In this paper, we consider the Laplacian on a smooth, strictly convex and bounded planar domain $\Omega$:
	\begin{align}
		\begin{cases}
			-\Delta u = \lambda^2 u & x \in \Omega\\
			B u  = 0,&x \in \d \Omega,
		\end{cases}
	\end{align}
	$B$ is a boundary operator corresponding to Dirichlet, Neumann or Robin boundary condtions. The even wave trace $w(t)$ is defined by the real part of the distribution
	\begin{align}\label{wt}
		 \frac{1}{2\pi} \int_{-\infty}^{+\infty} \Tr e^{it \sqrt{-\Delta}} \hat{\phi}(t) dt = \Tr \phi(\sqrt{- \Delta}), \qquad \phi \in \mathcal{S}(\R)
	\end{align}
	and we write $w(t) = \Tr \cos({ t \sqrt{-\Delta}})$. Denote by $\text{LSP}(\Omega)$ the (unmarked) length spectrum, consisting of lengths of periodic billiard trajectories in $\Omega$. The Poisson relation tells us that the singular support of $w(t)$ is contained in the closure of the length spectrum together with $0$ and its reflection about the origin. Assume $L$ is isolated in $\lsp(\Omega)$ with finite multiplicity and all corresponding periodic orbits are nondegenerate. Choose $\wh{\rho} \in C_c^\infty$ to be identically equal to $1$ in a neighborhood of $L$ and satisfy $\supp(\wh{\rho}) \cap \lsp(\Omega) = L$. We then have an asymptotic expansion of the regularized resolvent trace which has form
	\begin{align}\label{rrt}
		\int_0^\infty \wh{\rho}(t) w(t) e^{ik t} dt \sim \sum_{\gamma: \text{length}(\gamma) = L} \mathcal{D}_\gamma \sum_{j = 0}^\infty B_{\gamma, j} k^{-j},
 	\end{align}
	where the righthand side is a sum over all periodic orbits $\gamma$ having length $L$. We call $\mathcal{D}_{\gamma}$ the symplectic prefactor and $B_{\gamma, j}$ the Balian-Bloch wave invariants.
	\begin{theo}\label{main}
		Let $\Omega_\eps$, $\eps \in [0, \eps_0]$ be a smooth one parameter family of domains which fixes the reflection points and angles of a period $q \geq 3$ billiard orbit $\gamma$. Assume further that $\d^2 \LL$ has rank $q-1$ in $\Omega_0$, where $\d^2 \LL$ is the Hessian of the length functional in arclength coordinates evaluated at $\gamma$, and that the perturbation makes $\gamma$ nondegenerate in $\Omega_\eps$ when $\eps > 0$, with $|\det \d^2 \LL| \sim c_\gamma \eps$ for some $c_{\gamma} > 0$. Then, the $j$th Balian-Bloch wave invariant $B_{j, \gamma}$ associated to $\gamma$ in the regularized resolvent trace \ref{rrt} is given by
		\begin{align*}
			\frac{e^{- \frac{i \pi q}{4}}}{2q}  L_\gamma \left(\sum_{i_1, i_2, i_3 = 1}^q h_{i_1} h_{i_2} h_{i_3} \d_{i_1, i_2, i_3}^3 \LL  \right)^{2j} (c_\gamma \eps)^{-3j} \prod_{\ell = 1}^q \frac{\cos \theta_\ell }{|x(s_\ell) - x(s_{\ell+1})|} \frac{ (\pm i )^j}{w(j)} + \mathcal{R}_j,
		\end{align*}
		where
		\begin{itemize}
			\item $|x(s_\ell) - x(s_{\ell+1})|$ is the length of the $\ell$th link in the billiard trajectory.
			\item $\theta_\ell$ are the angles of reflection measured from the interior normal at $\ell$th impact point on the boundary.
			\item  $w(j) = \sum_{\mathcal{G}} \frac{1}{w(\mathcal{G})} > 0$ is a sum is over all $3$-regular graphs on $2j$ vertices, with $w(\mathcal{G})$ being the order of the automorphism group of a graph $\mathcal{G}$. In particular, it is independent of both the domain and the orbit.
			\item $\LL$ is the length functional in arclength coordinates (see Definition \ref{LF}), $L_\gamma$ is the length of $\gamma$ and the functions $h_{i} \in C^\infty(\d \Omega_0)$ are such that $h_{i_1} h_{i_2}$ is, up to a sign, the $(i_1,i_2)$ entry of the adjugate matrix of $\d^2 \LL$ at $\gamma$ in $\Omega_0$, which is independent of the deformation. The factor
			$$
			\sum_{i_1, i_2, i_3} h_{i_1} h_{i_2} h_{i_3} \d_{i_1, i_2, i_3}^3 \LL
			$$
			is, modulo an $O(\eps)$ error, the third directional derivative of $\LL$ along the null space of $\d^2 \LL$ in $\Omega_0$.
			\item The signs $\pm$ are given by $+$ if the closest eigenvalue to $0$  of $\d^2 \LL$ in $\Omega_{\eps}$ is positive and $-$ otherwise.
			\item $\mathcal{R}_j = \mathcal{R}(\gamma, \eps, j, \ka, \ka^{(1)}, \cdots, \ka^{(2j)})$, with $\ka^{(i)}$ being the $i$th derivative of the boundary curvature of $\Omega_{\eps}$, is a remainder satisfying:
			\begin{itemize}
				\item $\mathcal{R}_j = O(\eps^{-3j + 1})$.
				\item $\eps^{3j - 1} \mathcal{R}_j$ is smooth in all parameters down to $\eps = 0$.
				\item $\mathcal{R}_j$ depends on at most $2j$ derivatives of the boundary curvature in a small neighborhood of the reflection points of $\gamma$.
			\end{itemize} 
			\item When $j = 0$, $w(0) = 1$ and $\mathcal{R}_0 = 0$.
		\end{itemize}

		The symplectic prefactor associated to $\gamma$ is
		\begin{align*}
			\mathcal{D}_\gamma(k) =  \frac{e^{i k L_\gamma} e^{\frac{i\pi \sgn \d^2 \LL(S_\gamma)}{4}}}{\sqrt{|\det \d^2 \LL(S_\gamma)|}}.
		\end{align*}
	\end{theo}

	Our paper is part of a series in which we aim to show that the singular support of the wave trace and the length spectrum are in general distinct objects which encode different information. In part II, we will create cancellations in the wave trace by perturbing an initial domain having the properties in Theorem \ref{main} and then using the asymptotic formulas above together with a careful matching of Maslov indices, which arise as complex phases. The possibility of a smooth wave trace was first posed in \cite{DuGu75} and it was later remarked by Colin de Verdiere (see \cite{ZelditchSurvey2}) that apriori, there is no immediate reason why $w(t)$ cannot be $C^\infty$ on $\R \backslash \{0\}$; it is known to be singular at $t = 0$. In a work in progress, Hezari and Zelditch construct domains for which pairs of bouncing ball orbits give infinite order cancellations in the wave trace, going further in this direction. Our result applies to a large class of convex domains, having a rather general class of degenerate orbits; see Section \ref{PT}.
	
	
	

	
	\section{Background}
	Historically, the wave trace invariants have been exploited in the context of the inverse spectral problem, first posed by Mark Kac in \cite{Kac66}: ``Can one hear the shape of a drum?'' Mathematically, this asks if one can recover the geometry of a domain from knowledge of its Laplace spectrum. In general the answer is no, as was shown in \cite{Milnor}, \cite{GordonWebbWolpert}, \cite{Vigneras}, and \cite{Sunada1985RiemannianCA} amongst others. However, there are many cases in which one expects the answer to be yes, perhaps the most tantalizing of which is the class of strictly convex, smooth planar domains, which we consider in the present article. There have been numerous works in this direction, perhaps the most significant of which are \cite{Zel09} and \cite{HeZe19}. In the first article, Zelditch showed that generic families of analytic $\Z_2$ symmetric planar domains can be recovered from their spectrum by a careful analysis of the wave invariants associated to iterates of a bouncing ball orbit. We follow parts of this approach closely in our paper. In the second, Hezari and Zelditch prove that ellipses of small eccentricity are spectrally determined amongst all smooth, strictly convex bounded domains by their spectrum. This is the first and only known example of noncircular domains which are spectrally determined. $C^\infty$ compactness of Laplace isospectral sets was shown in \cite{POS0}, \cite{POS1} and \cite{POS2}. Other recent progress can be found in \cite{Zelditch1}, \cite{Zelditch2}, \cite{Zelditch3}, \cite{Zelditch5}, \cite{HeZe12}, \cite{Vig18}, and \cite{Vig23}.
	\\
	\\
	The dynamical analogue of the inverse spectral problem is to determine the shape of a manifold with boundary from knowledge of the length spectrum, consisting of lengths of closed geodesics (billiard trajectories) which make elastic reflections at the boundary. For planar domains, recent progress in this direction was acheived by the first author and Sorrentino in \cite{KaSo16}, together with the works \cite{KaAvDS16}, \cite{Koval} and \cite{KaDSWe17}. In the first three, it was shown that ellipses are isolated within the class of integrable domains. The latter two concern spectral rigidity of ellipses and nearly circular domains, respectively. $C^\infty$ compactness of marked length isospectral sets was demonstrated by the third author in \cite{Vig24}. Other significant results, including cases without boundary, are contained in \cite{KaHuSo18}, \cite{PopTop12}, \cite{Popov1994}, \cite{PopovTopalov}, \cite{GuKa}, \cite{CdVLSP} and \cite{GuillarmouLefeuvre}.
%
	
	\section{Billiards}\label{Billiards}
	\noindent Before obtaining a singularity expansion for the wave trace, we first review the relevant background on billiards. This will be useful in our construction of the Balian-Bloch resolvent parametrix in Section \ref{layer potentials}. We denote by $\Omega$ a bounded and strictly convex region in $\R^2$ with smooth boundary. This means that the curvature of $\d \Omega$ is a strictly positive function. We can identify $B^* \partial \Omega$ with ${\R}\slash {\ell \Z} \times (0, \pi)$, where $\ell=|\partial \Omega|$ is the length of the boundary. Let $x: [0,\ell] \to \R^2$ be a unit speed parametrization of $\d \Omega$ and denote by $ds$ the arclength measure on $\d \Omega$. The billiard map is defined on the coball bundle of the boundary, $B^* \d \Omega = \{s, \sigma) \in T^*\d \Omega : |\zeta| < 1 \}$, which can be identified in two ways with the cosphere bundle $S_{\d \Omega}^* \R{^2}$ over the boundary via the two orthogonal projection maps from inward $(+)$/outward $(-)$ pointing covectors. For $\sigma \in B^*\d \Omega$, define $\phi_{\pm}$ to be the inward $(+)$/outward $(-)$ facing covectors in $S_{\d \Omega}^* \R^2$ which project to $\sigma$.
	\\
	\\
	Define the maps
	\begin{align*}
		\beta_\pm(s, \sigma) = (s', \sigma'),
	\end{align*}
	where $s'$ is the arclength coordinate of the subsequent intersection of the line through $x(s)$ with direction $\phi_{\pm}$ and $\sigma$ is the projection onto $B^* \d \Omega$ of the parallel transport of $\phi_{\pm}$ to $x(s')$. We call $\beta = \beta_+$ the billiard map. It is well known that $\beta$ preserves the canonical one form $\sigma ds$ induced on $B^*(\d \Omega)$ and is differentiable there, extending continuously up to the boundary, with a square root type singularity at $0$ and $\pi$ corresponding to a Whitney fold in its graph. As a consequence, $\beta$ also preserves the area form $d \sigma \wedge ds$. The maps $\beta_\pm^{n}$ are defined via iteration and it is clear that $\beta_\pm^{-n} = \beta_\mp^n$ for $n \in \Z$. It is also apparent that $\frac{\d \beta}{\d \theta} > 0$, which means that vertical fibers in $B^* \d \Omega$ are twisted upon iterating $\beta$. Together, these properties make $\beta$ into an exact symplectic twist map. It has a particularly simple generating function
	\begin{align*}
		h(s,s') = - |x(s) - x(s')|,
	\end{align*}
	by which mean that $\beta\left(x, -\frac{\d h}{\d s}\right) = \left( s,  \frac{\d h}{\d s'}\right)$, which implies that $s$ and $s'$ give local coordinates on the graph of $\beta$.
	Geometrically, a billiard orbit corresponds to a union of line segments which are called links. A smooth closed curve $\mathcal C$ lying in $\Omega$ is called a {caustic} if any link drawn tangent to $\mathcal C$ remains tangent to $\mathcal C$ after an elastic reflection at the boundary of $\Omega$. By elastic reflection, we mean that the angle of incidence equals the angle of reflection at an impact point on the boundary. We can map $\mathcal C$ onto phase space $B^* \partial \Omega$ to obtain a smooth closed curve which is invariant under $\beta_\pm$. If the dynamics are integrable, these invariant curves are precisely the Lagrangian tori which folliate phase space. A point $P$ in $B^*\partial \Omega$ is called $q$-periodic ($q \geq 2$) if $\beta^q(P)=P$. We define the rotation number of a $q$-periodic orbit $\gamma$ emanating from $P \in \overline{B^*(\d \Omega)}$ by $\omega(\gamma) = \omega(P) = \frac{p}{q}$, where $p$ is the winding number of $\gamma$, which we now define.
	There exists a unique lift $\widehat{\beta}$ of the billiard map $\beta$ to the closure of the universal cover $\R \times [0,\pi]$ which is continuous, $\ell$ periodic in first variable and satisfies $\widehat{\beta}(\wt{s},0) = (\wt{s},0)$. Given this normalization, for any point $(\wt{s}, \wt{\sigma}) \in \R/\ell \Z \times [0,\pi]$ in the lift of a $q$ periodic orbit of $\beta$, we see that $\widehat{\beta}^q(\wt{s},\wt{\sigma}) = (\wt{s} + p \ell, \wt{\sigma})$ for some $p \in \Z$. We define this $p$ to be the winding number of the orbit generated by $\pi (\wt{s},\wt{\sigma}) = (s, \sigma) \in \overline{ B^* \d \Omega}$. Even if a point $(s,\sigma)$ generates an orbit which is not periodic in the full phase space but is such that $\pi_1 (\beta^q(s, \sigma)) = s$ for some $q \in \Z$, we can still define a winding number in this case. Such orbits are called $(p,q)$ geodesic loops, or loops for short. For a given periodic orbit, the winding number is independent of which point in the orbit is chosen, so we sometimes write $\omega(\gamma) = \omega(P')$ for any $P' \in \{P, \beta(P), \cdots, \beta^{q-1}(P)\}$. For deeper results and a more thorough introduction to the theory of dynamical billiards, we refer the reader to \cite{TabachnikovBilliards}, \cite{SiburgPrincipleofleastaction}, \cite{Katok}, \cite{Popov1994} and \cite{PopovTopalov}.

	\subsection{Parametrization and Notation}
	
	Following the notation in \cite{MaMe82}, we may translate and rotate $\Omega$ by a rigid motion so that there exists a point $p \in \d \Omega$ at the point $(0,0)$ with unit tangent vector $(1,0)$. Let $s$ denote the arclength coordinate along $\d \Omega$, measured counterclockwise from $p$. We then define the arclength parametrization
	\begin{align}\label{curvature coord}
		x(s) = (x_1(s), x_2(s)) = \int_0^s (\cos \phi(t), \sin \phi(t)) dt,
	\end{align}
	where $\phi$ is the angle made by the tangent to a base point, $dt$ is arclength on $\d \Omega$ and $\phi ' (t) = \kappa(t)$ is the curvature of $\d \Omega$. When dealing with successive points of reflection in the length functional, we use the capital letters $S = (s_1, \cdots, s_q)$  to denote the arclength coordinates of the reflection points and $x(S)$ to denote the corresponding boundary points $(x(s_1), \cdots, x(s_q))$ in $\R^2$, with $x_i(S)$ being the $i$th coordinate. Similarly, when dealing with a periodic orbit $\gamma$, we denote by $S_\gamma$ the arclength coordinates of the reflection points and by $\d \gamma$ the reflection points in $\R^2$.

	
	\subsection{Lengths of Orbits} 
	

	There are two natural functions for studying the length spectrum. If $\gamma$ is a $q$-periodic orbit of  the billiard map, the length of $\gamma$ is then defined to be
	\begin{align}
		L_\gamma = \sum_{i = 1}^q |x_{i+1} - x_i|,
	\end{align}
	where $x_i$ are the reflection points of $\gamma$ on the boundary, $x_{q+1} = x_1$ and $| \cdot |$ is the Euclidean distance function.
	
	\begin{def1}\label{LF}
		For an ordered collection of points $S = (s_1, \cdots, s_q) \in [0,\ell)^q$, we define the length functional $\LL(S)$ to be $\sum_{i = 1}^q |x_{i+1}(S) - x_{i}(S)|$, with $x_{q+1} = x_1$. For any $p, q \in \Z_{>0}$ relatively prime, we define the $(p,q)$ loop function $\ell_{p, q}(s)$ to be the length of the locally unique broken geodesic of rotation number $p/q$ based at $x(s)$, assuming a well defined branch of it exists (for example, when $p = 1$ and $q$ is sufficiently large; see Theorem 3.1 in \cite{Vig23}).
	\end{def1}

	In the same way that $h(s,s') = - |x(s) - x(s')|$ is a generating function for $\beta$, we have that $\LL(s_1, \cdots, s_q)$ is a generating function for $\beta^q$. The same holds $\ell_{p, q}$, at least locally in phase space.
	\begin{def1}
		The length spectrum of $\Omega$, denoted by $\lsp(\Omega)$, is the union of all lengths of periodic orbits.
	\end{def1}
	We will also need the notion of degeneracy for periodic orbits in the Section \ref{stationary phase and feynman diagrams}, where we apply the method of stationary phase with the length functional as the phase function of an oscillatory integral.

	\begin{def1}\label{ND}
		We say that a periodic orbit $\gamma$ is nondegenerate if $\det(1 - P_\gamma) \neq 0$, or equivalently if $\d^2 \LL(S_\gamma) \neq 0$ (see \ref{prod} and \ref{prod2} below).
	\end{def1}

	\subsection{Hessian Invariants}
	Both generating functions $\LL$ and $\ell_{p, q}$ have second order invariants associated to them at a critical point corresponding to a periodic orbit $\gamma$. The first is the Maslov index, which is a universal affine function of $\text{sgn} \d^2 \LL \mod 8$ coming from the complex exponential of the signature the Hessian upon performing a stationary phase expansion in \ref{BB}. From a symplectic geometry perspective, they are intersection numbers of vertical fibers upon iteration by $\beta$, corresponding to the twist property mentioned above. It  is related to the Morse index for the length variational problem with periodic boundary conditions. References can be found in \cite{Bialy} and \cite{WunschYangZou}. From a microlocal point of view, these indices arise as unitary phases when expanding via stationary phase a Lagrangian distribution associated to the quantization of the billiard map. They come from the integer valued exponents in the unitary transition functions for the Arnold-Keller-Maslov line bundle, sections of which, when tensored with a half density, invariantly define the notion of a principal symbol for Fourier integral operators. Their relavence in the trace formula comes from the fact that the wave propagator is a Fourier integral operator microlocally near nonglancing, transversally reflected rays. The nonpositivity of these complex phases will be important for cancelling contributions to the wave trace of different orbits in our subsequent paper.
	\\
	\\
	The other second order invariant is the determinant of $\d^2 \LL$. It is shown in \cite{KozTresh89} (Theorem 3, page 67) that
	\begin{align}\label{prod}
		\det \d^2 \LL = (-1)^{q+1} |\det(\Id - P_\gamma)| \prod_1^q b_i,
	\end{align}
	where $P_\gamma$ is the linearized Poincar\'e map of $\beta$ and
	\begin{align}\label{prod2}
		b_i = \frac{\d^2 |x_i(S) - x_{i+1}(S)|}{\d s_i \d s_{i+1}} = \frac{ \cos \theta_i \cos \theta_{i+1}}{|x_i(S) - x_{i+1}(S)|}.
	\end{align}
	The $b_i$ above consist of cosines of angles of reflection for the billiard orbit $\gamma$ at which $\LL$ is evaluated and $|x_i(S) - x_{i+1}(S)|$ is the length of a link connecting the $i$th and $i + 1$st reflection points in $\gamma$. The angles are measured from the inward pointing normal vector at the boundary. The Poincar\'e map will appear naturally in the symplectic prefactor for our stationary phase expansion of the regularized resolvent trace and the product \ref{prod} enters via the multiple reflection expansion with the Hankel functions in Proposition \ref{Hankel} below.
	
	\begin{rema}
		The use of $\LL : \d \Omega^q \to \R$ as a phase function comes naturally from the layer potential formulas in Section \ref{Balian-Bloch Invariants Section}. In \cite{MaMe82}, \cite{Vig20} and \cite{HeZe19}, the $(p,q)$ loop function was used instead. Both are generating functions for the $q$-th iterate of the billiard map and each presents its own difficulties. One disadvantage of $\LL$ is that it has singularities on the diagonal which require regularization. It also has a large number of varaiables to keep track of together with Maslov indices. The loop function is simpler in that it depends only on the base point of an orbit, but is in general multivalued and only local branches exist. Here, we primarily use $\LL$ as it is globally well defined and works with potentially smaller $q$.
	\end{rema}

	
	\section{Balian-Bloch Invariants}\label{Balian-Bloch Invariants Section}
	Let $\Omega$ be a compact, strictly convex domain having smooth boundary and a nondegenerate periodic orbit $\gamma$ of period $q$ and length $L_\gamma$, which we assume to be isolated in the length spectrum. The connection between the length and Laplace spectrums come from the Poisson relation:
	\begin{align*}
		\text{SingSupp} \tr \cos t \sqrt{-\Delta} \subset \{0\} \cup \pm \overline{\text{LSP}(\Omega)}.
	\end{align*}
	The lefthand side is the singular support of the distribution \ref{wt}, which is a purely spectral quantity, while the righthand side is geometric. In the case of bounded planar domains, this result is due to Anderson and Melrose \cite{AndersonMelrose}. If $\rho \in \mathcal{S}(\R)$ is such that $\hat{\rho} = 1$ in a neighborhood of $L_\gamma$ and $\supp \hat{\rho} \cap \pm \overline{\text{LSP}(\Omega)} = \{L_\gamma\}$, then the regularized resolvent at frequency $k$ is given by
	\begin{align}\label{BB}
		\int_0^\infty e^{i k t} \widehat{\rho}(t) w(t) dt = \frac{1}{i} \Tr \rho \ast k G_{\Omega, D}(k, x, y),
	\end{align}
	where $G_{\Omega, D}$ is the Green's kernel, i.e. the Schwartz Kernel of the resolvent $R_{\Omega, D}(k) = (-\Delta_{\Omega, D} - k^2)^{-1}$ on $\Omega$ with Dirichlet boundary conditions. This identity comes from the formula
	\begin{align*}
		R_{\Omega, D}(k) = \frac{i}{k} \int_0^\infty e^{i k t} \cos t \sqrt{- \Delta_{\Omega, D}} dt.
	\end{align*}
	The Dirichlet resolvent $R_{\Omega, D}$ on $\Omega$ extends meromorphically to $\C$, with poles at the spectrum of the Laplacian and corresponding residues equal to the orthogonal projectors onto finite dimensional eigenspaces. In view of the Paley-Weiner theorem, the regularized resolvent \ref{BB} is in fact an entire function of $k$. Replacing $k$ by $k + i \tau$ or $k+ i \tau \log k$, the maps $\tau \mapsto e^{ - \tau t} \wh{\rho}(t) w(t)$ and $\tau \mapsto k^{-\tau} \wh{\rho}(t) w(t)$ give rise to continuous families of compactly supported distributions. Consequently, the regularized resolvent trace has a well defined and smooth limit as $\tau \to 0^+$. When $\tau = 0$, we have the smoothed density of states formula
	\begin{align}
		\int_0^\infty e^{i k t} \wh{\rho}(t) w(t)dt \sim \pi \sum_{\pm} \sum_{n = 1}^\infty \rho(k \pm \lambda_n),
	\end{align}	
	where $\lambda_n^2$ are the eigenvalues of $- \Delta$, which provides a connection between the high frequency behavior of $- \Delta$ and the wave trace. From positivity of the spectrum, it is clear that the term in the sum with $\rho(k + \lambda_n) = O(k^{-\infty})$ is asymptotically negligible. The order of a singularity in the wave trace at a particular length is related to the decay properties of the correspondingly regularized resolvent trace. It follows that the expression \ref{BB} has an asymptotic expansion in $k$ of the form
	\begin{align}\label{exp}
		F_{\gamma}(k) \sum_{j = 0}^\infty \widetilde{B}_{\gamma, j} k^{-j},
	\end{align}
	where
	\begin{align}\label{SPF}
		F_\gamma(k) = (-1)^{\epsilon_B(\gamma)} \frac{c_0 L_\gamma^\# e^{ikL_\gamma} e^{i \pi m_\gamma / 4}}{\sqrt{|\det(I - P_\gamma)}|}
	\end{align}
	is called the principal symplectic prefactor. Here, $c_0$ is a universal constant, $L_\gamma = L$ the length of an isolated nondegenerate periodic orbit $\gamma$, $L_\gamma^\#$ the primitive period of $\gamma$, $m_\gamma$ the Maslov index of $\gamma$, $P_\gamma$ the linearized Poincar\'e map and $\eps_B(\gamma)$ depends on boundary conditions (see \cite{Zel09}). The coefficients $\widetilde{B}_{\gamma, j}$ are called the principal Balian-Bloch invariants associated to $\gamma$, sometimes also called wave trace invariants. The resolvent trace expansion was originally investigated by physicists Balian and Bloch in \cite{BB1}, \cite{BB2}, \cite{BB3}. The invariants $\wt{B}_j$ are clearly equivalent to usual wave trace invariants $a_{\gamma, k}$ associated to $w(t)$ near the length spectrum in Theorem \ref{aj} below. They are independent of the choice of $\rho$ subject to the constraints above, since only the singularity of $w(t)$ at $L_\gamma$ affects asymptotics of \ref{exp} upon expanding via stationary phase expansion the integral \ref{BB}.
	\\
	\\
	Our end goal is to show that one can find domains $\Omega$ and lengths of high multiplicity, such that several families of corresponding orbits have Maslov factors of opposite signs. We can then perturb $\d \Omega$ in such a way that the symplectic prefactors together with $\widetilde{B}_{\gamma, j}$'s cancel for arbitrarily large $j$. If the asymptotic expansion \ref{exp} is $O(k^{-\infty})$, then $w(t)$ is $C^\infty$ near $L$. If the decay is only $O(k^{-m})$ for some $m \geq 2$, then $w \in C^{m - 2}$ locally near $L_\gamma$. One difficulty arises from the nonreality of the coefficients $B_j$, whose complex phases may also depend on the corresponding orbit and could potentially negate whatever complex phase appears from the Maslov index. This is known not to be the case for the leading order invariant $B_0$, but not necessarily for $B_j$ with $j \geq 1$.
	\\
	\\
	We introduce the following notation for simplicity:
	\begin{def1}\label{modified}
		Define $\mathcal{D}_{B, \gamma}$ and $B_{\gamma, j}$ by the formula
		\begin{align*}
			F_{\gamma}(k) \sum_{j = 0}^\infty \widetilde{B}_{\gamma, j} k^{-j} = \mathcal{D}_{\gamma}(k) \sum_{j = 0}^\infty {B}_{\gamma, j} k^{-j},
		\end{align*}
		where
		\begin{align*}
			\mathcal{D}_{\gamma}(k) = \frac{c_0 e^{ikL_\gamma} e^{i \pi \text{sgn} \d^2 \LL/4}}{\sqrt{|\det \d^2 \LL|}}
		\end{align*}
		and the $B_{\gamma, j}$ are related to the $\widetilde{B}_{\gamma, j}$ as explained above.
	\end{def1}
The $B_{\gamma, j}$ are called modified Balian Bloch invariants and we will often refer to them simply as \textit{the} Balian-Bloch invariants when there is no risk of confusion. The reason for introducing $\mathcal{D}_{B, \gamma}$ and $B_{\gamma, j}$ is that they appear more naturally when expanding via stationary phase a microlocal parametrix for the resolvent in Section \ref{stationary phase and feynman diagrams}. The corresponding oscillatory integral for the regularized resolvent trace produces terms canonically equivalent to $|\det (\Id - P_\gamma)|$ and $m_\gamma$ in \ref{SPF}, but they appear more naturally in terms of the coordinates $x(S)$ introduced in Section \ref{Billiards}.

	\subsection{Relation to wave invariants}
We say that the length $L \in \R$ of a periodic orbit $\gamma$ is simple if up to time reversal ($t \mapsto -t$), $\gamma$ is the unique periodic orbit of length $L$. Without length spectral simplicity, there is no way to deduce Laplace spectral information from the length spectrum alone. It is shown in \cite{PeSt92} that generically, smooth convex domains have simple length spectrum and only nondegenerate periodic orbits. In that case, the following theorem holds:
\begin{theo}[\cite{GuMe79b}, \cite{PeSt17}] \label{aj}
	Assume $\gamma$ is a nondegenerate periodic billiard orbit in a bounded, strictly convex domain with smooth boundary and $\gamma$ has length $L$ which is simple. Then near $L$, the even wave trace has an asymptotic expansion of the form
	\begin{align*}
		\text{Tr} \cos t \sqrt{-\Delta} \sim \Re \left\{ a_{\gamma,0} (t - L + i0)^{-1} + \sum_{k = 0}^\infty a_{\gamma, k} (t - L + i0)^{k} \log(t - L + i0) \right\},
	\end{align*}
	where the coefficients $a_{\gamma k}$ are wave invariants associated to $\gamma$. The leading order term is given by
	\begin{align*}
		a_{\gamma, 0} =  \frac{e^{i \pi m_\gamma / 4}  L_\gamma^\#}{|\det (I - P_{\gamma})|^{1/2}}.
	\end{align*}
\end{theo}

Putting the formula from \ref{aj} into \ref{BB}, we recover the asymptotics in the Balian-Bloch expansion. An algebraic formula for one set of invariants in terms of the other can be found using basic identities in Fourier analysis.

	\subsection{Layer Potentials}\label{layer potentials}
	Here we fix a spectral parameter $\lambda  = k + i \tau \in \mathbb{C}_+$ with $\tau > 0$ and follow \cite{MET2}, \cite{Zel09} in describing potential theory for the resolvent. As before, we let $ds$ denote arclength along the boundary $\d \Omega$. From potential theory in the plane (see Proposition \ref{Resolvent}), we have the following formula for the interior Dirichlet resolvent:
	\begin{align}\label{potential}
		R_{\Omega, D}(\lambda) = \1_{\Omega} \left(R_0(\lambda) - 2 \Dl(\lambda)(\Id + \mathcal{N}(\lambda))^{-1}r_{\d \Omega} \Sl^\dagger(\lambda)\right) \1_{\Omega},
	\end{align}
	where
	\begin{align*}
		\Sl(\lambda) f(x) = \int_{\d \Omega} G_0(\lambda,x,y) f(y) ds(y), \qquad x \in \R^2 \backslash \d\Omega\\
		\Dl (\lambda) f(x) = \int_{\d \Omega} \d_{\nu_s} G_0(\lambda,x,y) f(y) ds(y), \qquad x \in \R^2 \backslash \d\Omega
	\end{align*}
	are the single and double layer potentials at frequency $\lambda$ and the boundary operator $\n(\lambda)$ is given by
	\begin{align}\label{N}
		\mathcal{N}(\lambda) f(x) = 2 \int_{\d \Omega} \d_{\nu_{s'}}  G_0(\lambda,x,y) f(y) ds(y). \qquad x \in \d\Omega
	\end{align}
	Here, $G_0(\lambda, x, y)$ is the free Green's function, i.e. the Schwartz kernel of the free outgoing resolvent on $\R^2$ at frequency $\lambda$, evaluated at $x \in \R^2$ and $y = y(s) \in \d \Omega$. The $\pm$ notation indicates limits taken from within $\Omega$ (resp. $\Omega^c$). The following jump formula holds across $\d \Omega$:
	\begin{align*}
		\Dl(\lambda) f_\pm = \frac{1}{2} (\pm \Id + \mathcal{N}(\lambda)) f.
	\end{align*}
	\begin{rema}
		Note that the signs of the double and boundary layer potentials depend on the choice of unit normal to the boundary. We will use the \textit{outward pointing} unit normal moving forward, for both the interior and exterior domains.
	\end{rema}
	The free Green's function on $\R^2$ is given by
	\begin{align}\label{i4hankel}
		G_0(\lambda, x,y) = \frac{i}{4} H_0^{(1)} (\lambda|x- y|),
	\end{align}
	where $H_\nu^{(1)}$ is a Hankel function of the first kind (of order $\nu$) given by $J_\nu^{(1)} + i Y_\nu^{(1)}$; $J_\nu^{(1)}, Y_\nu^{(1)}$ are Bessel functions of the first and second kind respectively. The free Greens function solves the problem
	\begin{align*}
		(- \Delta_{\R^2} - \lambda^2) G_0(\lambda, x, y) = \delta_0(x - y).
	\end{align*}
	\begin{rema}
		In even dimensions, the free resolvent $R_0(\lambda)$ is only holomorphic on the Riemann surface of the logarithm rather than all of $\C$, so taking $\lambda = k + i \tau \in \C$ with $\Im \lambda> 0$ allows us to consider only the the outgoing resolvent corresponding to $H_0(\lambda|x-y|)$ as opposed to $H_0(-\lambda|x-y|)$, since it is bounded on $L^2(\R^2)$. After using layer potentials to obtain an explicit parametrix for the interior Dirichlet resolvent on $\Omega$, we can let $\tau \to 0$ using meromorphy of $R_{\Omega, D}$ and holomorphy of the regularized resolvent trace.
	\end{rema}
	
 	 
	\begin{prop}\label{Hankel}
		For $\Omega \subset \R^2$ with the Euclidean Laplacian, we have
		\begin{align}
			\begin{split}
				\n(\lambda,x(s) , x(s')) \sim \lambda^{1/2}  e^{i \lambda |x(s) - x(s')| + 3 \pi i/ 4} |x(s) - x(s')|^{-1/2} \cos \theta a_1 (\lambda |x(s) - x(s')|),
			\end{split}
		\end{align}
		where $\theta$ is the angle made by the link with the interior unit normal at $x(s')$ and $a_1$ is a semiclassical symbol having the asymptotic expansion
			\begin{align*}
				a_1(z) &\sim \sum_{m = 0}^\infty i^m c_m z^{-m}, \qquad c_0 = \frac{1}{\sqrt{2\pi}},\\
				c_m &= \frac{(4 - 1^2) (4 - 3^2) \cdots (4 - (2m - 1)^2 )}{m! 8^m \sqrt{2\pi}}.
			\end{align*}
		The symbol $a_1(\lambda|x(s) - x(s')|)$ is holomorphic in $\lambda$ on $\C \backslash (-\infty, 0]$. Tor any fixed $\varrho \in (0,1/2)$ and $\chi \in C_c^\infty(\R)$ supported on $[-1,1]$, $(1 - \chi(k^{\varrho} z )) a_1((k + i \tau) z)$ belongs to the symbol class $S_\varrho^0(\R)$:
		\begin{align*}
			\left|\d_z^\alpha \d_k^\beta \left((1 - \chi(k^{\varrho } z))a_0((k+i\tau)z)\right) \right| \leq C_{\alpha, \beta, \varrho, \tau, K} \langle k \rangle ^{ \varrho |\alpha|  - |\beta| }
		\end{align*}
		for all $K \subset \subset \R$ compact and $\alpha, \beta, \tau$.
	\end{prop}

	\begin{proof}
		Using the well known formula
		\begin{align*}
			\frac{d}{dz} H_\nu^{(1)}(z) = \frac{\nu H_\nu^{(1)}(z)}{z} - H_{\nu + 1}^{(1)} (z),
		\end{align*}
		and differentiating $G_0$ to obtain the boundary layer operator, we obtain
		\begin{align*}
			\n(\lambda, x(s), x(s')) = - \frac{i \lambda}{2} \cos \theta H_1^{(1)}(\lambda|x(s) -x(s')|).
		\end{align*}
		The $\cos \theta$ term comes from the calculation
		\begin{align*}
			\nabla_y |x - y| = \frac{y- x}{|x - y|},
		\end{align*}
		which is the unit vector in the direction of the link $\overline{xy}$. The asymptotics of $H_1^{(1)} (\lambda |x(s) - x(s')|)$, including the coefficients $c_m$ and the phase $-3 \pi /4$, are given in \href{https://dlmf.nist.gov/10.17}{NIST}:
		\begin{align*}
			\frac{1}{2} H_1^{(1)}(z) \sim \frac{e^{i z - 3\pi i/4}}{\sqrt{z}} \sum_{m} i^m c_m z^{-m}.
		\end{align*}
		Combining with $-i$ gives a total phase of $+ 3\pi i / 4$.
		That $a_1 \in S_0^0$ follows immediately from the asymptotic expansion in the statement of the proposition. Differentiating term by term gives
		\begin{align*}
			\left|\d_z^\alpha \d_k^\beta a_0(kz) \right|  \leq  \sum_{m = \max{\{\alpha, \beta\}}}^\infty c_m \frac{(m!)^2}{((m - \alpha)! (m - \beta)!)}k^{- m - \beta }z^{-m - \alpha}  \lesssim  \langle k \rangle^{-\beta} z^{- \alpha},
		\end{align*}
		as $k \to \infty$, $z$ away from $0$. Here, $\langle k \rangle = (1 + k^2)^{1/2}$ is the Japanese bracket. Differentiating $(1 - \chi(k^{\varrho} z))$ clearly shows that it is in $S_\varrho^0$, so $(1 - \chi(k^{\varrho} z)) a_1(kz) \in S_\varrho^0$.
	\end{proof}
	
	\begin{rema}
		The cutoff factor was introduced to localize asymptotics away from zero, where the asymptotic expansion isn't useful. The introduction of the parameter $\varrho$ is important for the decomposition of $\n$ in Proposition \ref{SP1} below, in which the resolvent is decomposed into a sum of microlocal (homogeneous) pseudodifferential operators and semiclassical Fourier integral operators at scale $\hbar = 1/k$. This decomposition is only valid when working with symbol classes which have $\varrho > 0$.
	\end{rema}
	In calculations below, we will use the function $a_1$ in formula \ref{Hankel} to avoid extra combinatorial constants while keeping track of complex phases.
	
	\subsection{Reduction to the boundary}
	We now derive the way in which layer potentials translate the interior problem to one defined only on the boundary, following the presentation in \cite{Zel0res}. To do so, we adapt the Poisson integral solution of the boundary value problem considered in \cite{MET2} to the Dirichlet/Neumann resolvents:
	\begin{prop}\label{Resolvent}
		The Dirichlet resolvent on $\Omega$ and the Neumann resolvent on the exterior domain $\Omega^c$ are given by
		\begin{align*}
			R_{\Omega,D}(\lambda) =& \1_{\Omega} \left(R_0(\lambda) + 2\Dl(\lambda) (\Id + \n_{\Omega}(\lambda))^{-1} r_{\d \Omega} R_0(\lambda)\right) \1_{\Omega},\\
			R_{\Omega^c, N}(\lambda) =& \1_{\Omega^c} \left(R_0(\lambda) - 2 \Sl(\lambda) (\Id - \n_{\Omega^c}^\# (\lambda)))^{-1} r_{\d \Omega^c} \d_{\nu_{\Omega^c}} R_0(\lambda)\right) \1_{\Omega^c},
		\end{align*} 
		where
		\begin{align*}
			\n_{\Omega}^\# g(x) = 2 \int_{\d \Omega} \d_{\nu_{x, \Omega}} G_0(\lambda, x ,y) g(y)ds(y), \qquad x \in \d \Omega\\
			\n_{\Omega^c}^\# g(x) = 2 \int_{\d \Omega^c} \d_{\nu_{x, \Omega^c}} G_0(\lambda, x ,y) g(y)ds(y). \qquad x \in \d \Omega^c
		\end{align*}
	\end{prop}
	\begin{proof}
		We begin with solutions for the boundary value problems given in \cite{MET2}, keeping in mind our convention on exterior normals. For brevity, put $P(\lambda) = - \Delta - \lambda^2$:
		\begin{align*}
			&\begin{cases}
				P(\lambda) u(x) = 0,\\
				r_{\d \Omega} u = f,
			\end{cases}
		\implies
		u =  - 2 \Dl(\lambda) (\Id + \n(\lambda))^{-1} f, \qquad (\text{interior Dirichlet})\\
			&\begin{cases}
			P(\lambda) v(x) = 0,\\
			r_{\d \Omega^c} \d_\nu v = \phi,
		\end{cases}
		\implies
		v = 2 \Sl(\lambda) (\Id - \n_{\Omega^c}^\#(\lambda))^{-1} \phi. \qquad (\text{exterior Neumann})
		\end{align*}
		Notice that $\n_{\Omega}^\#$ is defined with the outward pointing normal for $\Omega^c$, and hence the inward pointing normal for $\Omega$. To obtain the resolvents, set $\wt{G} = R_{\Omega, D}(\lambda) - R_0(\lambda)$ and $\wt{\wt{G}}(\lambda) = R_{\Omega^c, N} (\lambda) - R_0(\lambda)$. These solve the interior Dirichlet and exterior Neumann problems with $f = - r_{\d \Omega} R_0(\lambda)$ and $\phi = - r_{\d \Omega} \d_\nu R_0(\lambda)$. Applying the layer potentials to each and adding back $R_0$ completes the proof.
	\end{proof}
	To combine the interior and exterior resolvents into a similar form, note that for the exterior problem, the outer normal points in the opposite direction and $\n_{\Omega} (\lambda) = - \n_{\Omega^c} (\lambda)$. The symmetry $G_0(\lambda, x, y) = G_0(\lambda, y, x)$ coming from the explicit form of the resolvent in \ref{i4hankel} implies that  $\n_{\Omega}^\# = \n_{\Omega}^\dagger$ and $\n_{\Omega^c}^\# = \n_{\Omega^c}^\dagger$. We also have that $(r_{\d \Omega^c} \d_\nu R_0(\lambda))^\dagger = \Dl(\lambda)$ and $ r_{\d \Omega} R_0(\lambda) = \Sl^\dagger(\lambda)$. The free resolvent $R_0(\lambda)$ is formally self adjoint, so taking the transpose of $R_{\Omega^c, N}(\lambda)$ and putting in cutoffs to the interior/exterior, we obtain
	\begin{align}\label{reduction}
	\begin{split}
		R_{\Omega,D}(\lambda) =& \1_{\Omega} \left(R_0(\lambda) + 2\Dl(\lambda) (\Id + \n(\lambda))^{-1} r_{\d \Omega} R_0(\lambda)\right) \1_{\Omega},
		\\
		R_{\Omega^c, N}^\dagger (\lambda) =& \1_{\Omega^c}  \left(R_0(\lambda) + 2\Dl(\lambda) (\Id + \n(\lambda))^{-1} r_{\d \Omega} R_0(\lambda) \right) \1_{\Omega^c},
	\end{split}
	\end{align}
	where $r_{\d \Omega}$ is the operator restricting to the boundary and now \textit{all} operators are defined using the outward pointing normal \textit{from} $\Omega$. Adding the two and subtracting the free resolvent yields
	\begin{align}\label{dsum}
		\tr \left( R_{\Omega, D}(\lambda)  + R_{\Omega^c, N}^\dagger(\lambda) - R_0(\lambda) \right) =  2 \tr \left( r_{\d \Omega} R_0(\lambda) (\1_{\Omega}^2 + \1_{\Omega^c}^2) \Dl(\lambda)(\Id + \n(\lambda))^{-1} \right),
	\end{align}
	where we commuted $r_{\d \Omega} R_0(\lambda)$ through the trace to obtain a kernel on $\d \Omega \times \d \Omega$. Note that the off diagonal terms are trace free. Observe also that the leftmost factor in the righthand side of \ref{dsum} has the property that
	\begin{align*}
		2 r_{\d \Omega} R_0(\lambda) \Dl (\lambda) &= 2 r_{\d \Omega}  R_0(\lambda) \circ \d_{\nu_2} R_0 (\lambda) r_{\d \Omega}\\
		&=  2 r_{\d \Omega}  \frac{1}{2 \lambda}\frac{d}{d \lambda} \d_{\nu_2}  R_0(\lambda) r_{\d \Omega} =  2 \frac{1}{2 \lambda} \frac{d}{d \lambda}\left( \frac{1}{2} \n(\lambda) \right),
	\end{align*}
	where we used the identity
	\begin{align*}
		\frac{d}{d \lambda} R_0(\lambda) = 2 \lambda R_0^2(\lambda),
	\end{align*}
	together with the fact that $\d_{\nu_2}$ is differentiating in only one of the variables. Hence, the righthand side of \ref{dsum} adds up to
	\begin{align}\label{logdet}
		\frac{1}{2 \lambda} \tr \frac{d }{d \lambda}\log (\Id + \n(\lambda)) = \frac{1}{2 \lambda} \frac{d}{d \lambda} \log \det (\Id + \n(\lambda)).
	\end{align}
	It is shown in \cite{ZworskiPoissonformulaforresonances}, \cite{SjostrandTraceformula} and \cite{ChristiansenResonances} that the regularized exterior resolvent trace $\Tr (R_{\Omega^c, N}(\lambda) - R_0(\lambda))$ admits an asymptotic expansion in negative powers of $\lambda$ for each periodic orbit in $\Omega^c$. By convexity of $\Omega$, there are only gliding orbits along $\d \Omega$ and no transversally reflected orbits in the exterior domain. In particular, for $\wh{\rho}$ supported away from $\Z | \d \Omega|$, we have
	\begin{align*}
		\int_0^\infty e^{itk} \wh{\rho}(t) \cos t \sqrt{- \Delta_{\Omega^c, N}} = O(k^{-\infty}).
	\end{align*}
	Convolving \ref{logdet} with the test function $- i \lambda \rho(\lambda)$ gives
	\begin{align}
		 \sum_\gamma \mathcal{D}_{\gamma}(k) \sum_{j = 0}^\infty {B}_{\gamma, j} k^{-j} \sim  \frac{1}{2 i} \int_{-\infty}^{\infty} \rho(k -  \mu) \frac{\d}{\d \mu} \log \det (\Id + \n(\mu + i \tau))^{-1} d \mu.
	\end{align}
	Verification that the relevant functional analytic properties are satisfied to make the above formal computations legitimate (e.g. trace class, Fredholm, etc.) are contained in \cite{Zel0res} and \cite{Zelditch3}.

	\subsection{Regularization}
	As in the introduction, let $\wh{\rho} \in C_c^\infty(\R)$ be identically equal to $1$ in a neighborhood of an isolated length $L \in \lsp(\Omega)$, and zero outside of a small interval $(L- \eps, L+ \eps)$ such that $\supp \wh{\rho} \cap \overline{\lsp} = \{L\}$. Convolving the regularized trace above with $\rho$ and expanding $\n(\lambda)$ in a finite Neumann series yields the following formula connecting the resolvent trace asymptotics to billiard dynamics:
	\begin{prop}[\cite{Zel09}, Proposition 3.6] \label{mf}
	Let $\chi_{\d \gamma}(x, \Re (\lambda)^{-1} D_x)$ be a semiclassical pseudodifferential operator microlocalizing in phase space near the projection of an orbit $\gamma$ having length $L$ to $B^* \d \Omega$. Then, for each $J \in \N^+$, there exists $M_0(J)$ such that
	\begin{align}
		\begin{split}
			\Tr \int_{\R} \rho( k - \zeta)N^{M_0+1}(\zeta + i \tau)(\Id + \n(\zeta + i \tau))^{-1} \n'(\zeta + i \tau ) \chi_{\d \gamma}(\zeta) d\zeta = O(k^{-(J+1)}).
		\end{split}
	\end{align}
	In particular,
	\begin{align}
		\begin{split}
			\frac{1}{2i} \sum_{M = 0}^{M_0} (-1)^{M} \Tr &\int_{\R} \rho( k - \zeta) \n^M(\zeta + i \tau) \n'(\zeta + i \tau )\chi_{\d \gamma}(\zeta) d\zeta\\
			&= \mathcal{D}_{B, \gamma}(k + i \tau) \sum_{j = 0}^J B_{\gamma, j} k^{-j} + O(k^{-(J+1)}).
		\end{split}
	\end{align}
\end{prop}

	This reduces to the computation of $B_{\gamma, j}$ to a calculation of boundary integrals corresponding to powers of $\n$. The boundary operator $\n$ does not have small operator norm and hence the infinite series isn't actually convergent. However, the remainder can still be made arbitrarily small for sufficiently large $M$, which is shown in \cite{Zelditch3}. This explains that while not an asymptotic expansion in the usual case, the Neuman series still provides an effective algorithm for computing the coefficients $B_j$. In \cite{BB1}, \cite{BB2}, \cite{BB3}, the physicists Balian and Bloch refer to the sum above as the multiple reflection expansion. The kernel $\n^M(\lambda,s,s')$ is of the form
	\begin{align*}
		\n^M(\lambda,x(s),x(s')) &= \int_{\d \Omega^{M-1}} \n(\lambda, s, s_1) \n(\lambda, x(s_1), x(s_2)) \cdots \n(\lambda, x(s_{M-1}), x(s')) ds_1 \cdots ds_{M-1}\\
		& = \int_{\d \Omega^{M-1}} e^{i k \LL (x(s), x(s_1), \cdots, x(s_{M-1}) , s')} a(\lambda, s,s_1, \cdots, s_{M-1}, s')ds_1\cdots ds_{M-1},
	\end{align*}
	for the semiclassical amplitude
	\begin{align}
		a(\lambda, s, s_1, \cdots, s_{M-1}, s') = \lambda^{M/2} e^{3 \pi i M/4} \prod_{i = 1}^{M} \frac{ \cos \theta_i}{| x(s_{i-1}) -  x(s_i)|^{\frac{1}{2}}} a_1(\lambda| x(s_{i-1}) -  x(s_i)|).
	\end{align}
	The term \textit{multiple reflection} refers to the phase function, which is a sum of consecutive generating functions for $\beta$ and hence generates the $M$-fold iterate of the billiard map. Here, $a_1$ is the semiclassical symbol in Proposition \ref{Hankel} and we use the convention that $s = s_0, s' = s_M$. Integrating by parts the formula in Proposition \ref{mf} and reindexing, one obtains
	\begin{align*}
		\begin{split}
			\frac{1}{2i} \sum_{M = 1}^{M_0 + 1} \frac{ (-1)^{M}}{M}  \Tr &\int_{\R} \rho'( k - \zeta)\n^{M}(\zeta + i \tau) d\zeta\\
			&= \mathcal{D}_{B, \gamma}(k + i \tau) \sum_{j = 0}^J B_{\gamma, j} k^{-j} + O(k^{-(J+1)}). 
		\end{split}
	\end{align*}
	With the understanding that the following is not a true asymptotic expansion, but rather an algorithm for computing the wave invariants $B_{\gamma, j}$ as above, we write
		\begin{align}\label{mf2}
		\begin{split}
			\mathcal{D}_{B, \gamma}(k + i \tau) \sum_{j = 0}^\infty B_{\gamma, j} k^{-j} \sim^* \frac{1}{2i} \sum_{M = 1}^{\infty} \frac{ (-1)^{M}}{M}  \Tr \int_{\R} \rho'( k - \zeta) \n^{M}(\zeta + i \tau) d\zeta.
		\end{split}
	\end{align}
	The notation $\sim^*$ is meant to emphasize the fact that the expansion is in the sense of Proposition \ref{mf}. The operator $\n$ is both a \textit{microlocal} pseudodifferential operater and a \textit{semiclassical} Fourier integral operator. It has microlocally homogeneous singularities near the diagonal and away from the diagonal, it is a semiclassical Fourier integral operator quantizing the billiard map with semiclassical parameter $\hbar = 1/k$. Hence, some care must be taken to apply the $\Psi$DO and FIO calculi. In \cite{HaZe}, Hassell and Zelditch introduced a microlocal decomposition $\n = \n_0 + \n_1$,
	\begin{align*}
		\n_0(x(s), x(s')) &= \chi(k^{\rho} |x(s) - x(s')) \n(\lambda, x(s),x(s')),\\
		\n_1 (x(s), x(s')) &= (1- \chi(k^{\rho} |x(s) - x(s')|)) \n(\lambda, x(s), x(s')),
	\end{align*}
	where $\n_0 \in \Psi^{-1}(\d \Omega)$ and $\n_1 \in I_{1/k}^0(\d \Omega \times \d \Omega; \Lambda_\beta)$ is a Lagrangian distribution of order zero which is the Schwartz kernel of a semiclassical Fourier integral operator quantizing the billiard map. Here, $\Lambda_\beta$ is the twisted graph of the billiard map:
	\begin{align*}
		\Lambda_\beta = \{(s, s', \sigma, - \sigma'): (s, \sigma), (s', \sigma') \in B^*(\d \Omega), \beta(s, \sigma) = (s', \sigma')\},
	\end{align*}
	thought of as a Lagrangian submanifold of the product cotangent bundle. In this case, one has
	\begin{align*}
		\n^M = (\n_0 + \n_1)^M = \sum_{\sigma: \Z_M \to \{0,1\}} \n_{\sigma(0)} \n_{\sigma{(1)}} \cdots \n_{\sigma(M)}.
	\end{align*}
	We denote these compositions by $\n_\sigma$ and write $|\sigma| = \#\{i: \sigma(i) = 0\}$. The composition rules for Fourier integral operators tell us that the product $\n_0 \n_1$ has canonical relation contained in $\Delta_{\d \Omega} \circ \Lambda_\beta = \Lambda_\beta$, and the fact that $\n_0 \in \Psi^{-1}$ reduces the order of the symbol by $1$. Hence, $\n_\sigma$ has order $-|\sigma|$ and quantizes the $M - |\sigma|^{th}$ iterate of the billiard map. As we have microlocalized near an orbit $\gamma$ with primitive period $q$, only the terms with $M - |\sigma| = rq$ for $r \in \Z$ potentially contribute to the leading order asymptotics near the lengths $rL$ of a $p/q$ periodic orbit iterated $r$ times. This is shown rigorously in \cite{Zel09}, where it is also demonstrated that for any specified order $R$ and sufficiently large corresponding $M$, the terms $\n_0^M$ do not contribute to the regularized resolvent trace asymptotics modulo $O(k^{-R})$. In addition, it is shown there that if $L = L_\gamma$ is simple, then the trace is unchanged modulo $O(k^{-\infty})$ when the semiclassical cutoff in \ref{mf} is removed. The presence of a microlocal cutoff is what is meant in Theorem \ref{main} regarding the contributions to the resolvent trace of a specific orbit $\gamma$. The reason for a small remainder in Proposition \ref{mf} is that after microlocalizing near the projection of $\gamma$ to $B^*(\d \Omega)$ and convolving with $\rho$, $\wh{\rho}$ being supported near $L_\gamma$, $\n_\sigma$ can have at most $q$ factors of $\n_1$ without making the phase nonstationary. For large enough $M$, at least $M -q \gtrsim R$ factors of $\n_0$ must enter the composition, each of which reduces the order by $1$.
	\\
	\\
%
%
	The only possible terms contributing to $B_{\gamma, j}$ have $q \leq M \leq q + j$; if $M$ were larger, the terms would be at most of order $k^{-j -1}$ and if $M$ were smaller than $q$, the phase would be nonstationary on the support of $\widehat{\rho}$, in which case the asymptotics would be $O(k^{-\infty})$. Moreover, we will see later that only the compositions with $|\sigma| = 0$ generate useful terms containing
	\begin{enumerate}
		\item highest powers of a yet to be specified deformation parameter $\epsilon^{-1}$ and
		\item highest order derivatives of the boundary curvature upon performing a stationary phase expansion.
	\end{enumerate}


\subsection{Principal symbol calculation}
	To compute \ref{mf2}, we use the Fourier inversion formula to see that
	\begin{align}\label{IFT}
		\begin{split}
			&\tr  \int \rho' (k - \zeta) \n_\sigma (\zeta + i \tau) d\zeta = \frac{1 }{2\pi} \int_{{\d \Omega}_S^{M }} \int_{-\infty}^{\infty}    \int_{-\infty}^{\infty} e^{i(k - \zeta)t} i t  \hat{\rho}( t) \\
			& \times \prod_{i = 1}^{M}  e^{i(\zeta + i \tau)|x(s_i) - x(s_{i+1})| +3 \pi i / 4} \chi_i(\zeta^{\varrho} |x(s_i) - x(s_{i+1})|)\\ 
			& \times \sqrt{\frac{(\zeta + i \tau)}{|x(s_i) - x(s_{i+1})|}}   a_1((\zeta + i \tau) |x(s_i) - x(s_{i+1})|) \cos \theta_i dt d \zeta d S,
		\end{split}
	\end{align}
	where $S = (s_1, \cdots, s_{M})$, $s_{M +1} = s_1$ and
	\begin{align*}
		\chi_i(z) = \begin{cases}
			\chi(z) & \sigma(i) = 0\\
			1- \chi(z) & \sigma(i) = 1,
		\end{cases}
	\end{align*}
	for a cutoff $\chi$ equal to $1$ on $[-1/2,+1/2]$ and zero outside of $[-1,+1]$ (cf. Proposition \ref{Hankel}). Changing variables $\zeta \mapsto \zeta/k$, formula \ref{IFT} can be written as the oscillatory integral
	\begin{align}\label{Nsig}
		\begin{split}
			k \int_{-\infty}^{\infty} \int_{-\infty}^{\infty}  \int_{{\d \Omega}_S^{M} } e^{i k \Phi({\zeta} + i {\tau}, S, {t})}  B_\sigma( k{\zeta} + i {\tau}, t, S)   d {\zeta} d{t} dS,
		\end{split}
	\end{align}
	where the phase is given by
	\begin{align}
		\Phi({\zeta} + i {\tau}, S, {t}) =  (1- \zeta)t +  {\zeta} \LL(S) 
	\end{align}
	and the amplitude by
	\begin{align}\label{phase}
		\begin{split}
				B_{\sigma} ( \lambda, S) = &
			 \frac{ e^{3 M \pi i / 4 } }{2 \pi } i t  \hat{\rho}({t}) \lambda^{\frac{M}{2}}e^{- \tau \LL(S)} \prod_{i = 1}^{M} \wt{\chi}_i(\lambda, S)\\
			&\times \frac{\cos \theta_i}{|x(s_i) - x(s_{i+1})|^{1/2}}  a_1(\lambda |x(s_i) - x(s_{i+1})|),
		\end{split}
	\end{align}
	where we use the notation
	\begin{align*}
		\wt{\chi}_i(\lambda, S) = \begin{cases}
			\chi( \Re(\lambda)^{\varrho} |x(s_i) - x(s_{i+1}) |) & \sigma(i) = 0\\
			1- \chi(\Re(\lambda)^{\varrho} |x(s_i) - x(s_{i+1})|) & \sigma(i) = 1.
		\end{cases}
	\end{align*}
	Recall that $\theta_i$ is the angle between the interior normal to the boundary and the link $(x(s_{i+1}) - x(s_i))$.	When we later set $|\sigma| = 0$ and $M = q$ to obtain leading order asymptotics, all of the $\chi_i$ will be one.
	
	
	\begin{rema}
		Note as above that a priori, $B_\sigma \in S_\varrho^{ \frac{M}{2}}(\d \Omega^{M})$. The condition $\varrho \in (0, \frac{1}{2})$ guarantees that one can integrate out the diagonal singularities corresponding to terms with $\sigma(i) = 0$ to obtain a Fourier integral kernel of lower order quantizing $\beta^{M - |\sigma|}$, as was done in \cite{Zel09}. We will make use of this in Section \ref{stationary phase and feynman diagrams}.
	\end{rema}
	An application of stationary phase in the variables ${\zeta}, {t}$ of \ref{Nsig} gives the following:
	
	\begin{prop}[\cite{Zel09}] \label{SP1} 
		For each $\sigma: \Z_M \to \{0,1\}$, $\Tr \rho' \ast \n_\sigma(k + i \tau)$ is given modulo $k^{-\infty}$ by
		\begin{align}
			e^{3 M \pi i/ 4} i \int_{{\d \Omega}_S^{M} }  e^{i (k + i \tau) \LL(S)} a_0^\sigma (k+i \tau, S) dS
		\end{align}
		where
		\begin{align*}
			a_0^\sigma (k +i \tau, S) = \hat{\rho}(\LL(S)) \left(\LL(S)  A_\sigma (k + i \tau, S) - \frac{1}{i}  \frac{\d}{\d k} A_\sigma (k + i \tau, S)  \right) \in S_\varrho^{\frac{M}{2}}(\d \Omega^M)
		\end{align*}
		and $A_\sigma ( k + i \tau, S)$ is
		\begin{align*}
			( k + i \tau)^{M/2}  \prod_{i = 1}^{M}  \chi_i(k + i \tau,S) \frac{\cos \theta_i}{ |x(s_i) - x(s_{i+1})| ^{1/2}}  a_1(( k + i \tau) |x(s_i) - x(s_{i+1})|).
		\end{align*}
		The complex phase is given by $(3M + 2)\pi/4$ and $a_1$ is the Hankel amplitude from Proposition \ref{Hankel}.
\end{prop}
	\begin{proof}
		See Equation \ref{Zk} in Section \ref{stationary phase and feynman diagrams} below for the stationary phase formula. In our case, the stationary points occur at $t = \LL, {\zeta} = 1$. The Hessian is given by
		\begin{align*}
			\d_{{\zeta}, {t}}^2 \Phi = \begin{pmatrix}
				0 & -1\\
				-1 & 0
			\end{pmatrix},
		\end{align*}
		from which we obtain $\langle \d_{{\zeta}, {t}}^2 \Phi^{-1} \d_{{ \zeta}, {t}}, \d_{{\zeta}, {t}} \rangle = - 2 \d_{{\zeta}} \d_{{t}}$. The signature of the Hessian is $0$ and the modulus of its determinant is one. The total prefactor is then $\frac{2\pi}{k} e^{i k  \LL}$, the coefficient of which cancels with that from Fourier inversion and the change of variables above. Since the phase is quadratic, only the amplitude is differentiated ($g = 0$, $\mu = 0$ in the notation of formula \ref{Zk} below). Since the amplitude factors into a locally linear function of $t$ and one of $\zeta$, there are only two nonzero terms. Near the critical point corresponding to the length $L_\gamma$ of a periodic orbit, $\hat{\rho}$ is constant and hence $\d_t$ only lands on $t$. The $k$ coming from $\d_{{ \zeta}}$ is canceled by the $k^{-1}$ in stationary phase as is the factor of $2\pi$. The factor $e^{- \tau \LL(S)}$ is unchanged by the Hessian operator and can be reinserted into the phase function.
	\end{proof}
	In the next section, we will prove that the $|\sigma| \geq 1$ terms do not contribute to the leading order deformation asymptotics or highest order derivatives in the curvature jet of the boundary. Hence, we can set $\chi_i = 1$ in the expression for $A_\sigma$ in \ref{SP1}. As $\n_0 \in \Psi^{-1}(\d \Omega)$, the symbol $a_0^\sigma$ in Proposition \ref{SP1} can be integrated to yield a symbol in $S_\varrho^{\frac{M - 3 |\sigma|}{2}}(\d \Omega^{M - |\sigma|})$; see Theorem 3.8 in \cite{Zel09}. Moreover, since $a_1 \in S_\varrho^{\frac{M}{2}}(\d \Omega^M)$, $A_\sigma' \in S_\varrho^{\frac{M}{2}-1}$. Hence, for $|\sigma| = 0$, $M = q$ and $S$ near $S_\gamma$, we have
		\begin{align}
			a_0^0 (\lambda, S) = \left(\frac{\lambda}{2\pi}\right)^{q/2} \LL(S) \prod_{i = 1}^{q}  \frac{\cos \theta_i}{ |x(s_i) - x(s_{i+1})| ^{1/2}}  + O_{S_\varrho^0} \left(\Re (\lambda)^{\frac{q}{2} - 1}\right),
		\end{align}
	where we used the leading order asymptotics for $a_1$ in formula \ref{Hankel}.


	\begin{rema}\label{PCR}
		Recall that by formula \ref{prod}, the product in the leading order expansion of $a_0$ is precisely
		\begin{align*}
			\prod_{i = 1}^{q}  \frac{\cos \theta_i}{ |x(s_i) - x(s_{i+1})| ^{1/2}} = \sqrt{\frac{\det \d^2 \LL}{|\det (1 - P_\gamma)|}},
		\end{align*}
		with $P_\gamma$ being the Poincar\'e map associated to $\gamma$.
	\end{rema}
	Combining Proposition \ref{SP1}	and \ref{mf2}, we obtain the following:
	\begin{coro}\label{fullexp} The coefficients in the regularized resolvent trace expansion can be determined from the oscillatory integrals in Proposition \ref{SP1}:
	\begin{align*}
		\mathcal{D}_{\gamma}(\lambda)   \sum_{j = 0}^\infty B_{\gamma, j} \Re(\lambda)^{-j} \sim^* \frac{1}{2} \sum_{M = 1}^\infty \sum_{\sigma: \Z_M \to \{0,1\}} \frac{e^{ -M \pi i / 4}}{M} \int_{{\d \Omega}_S^{M}}  e^{i \lambda \LL(S)} a_0^\sigma(\lambda, S) dS.
	\end{align*}
	\end{coro}
	

	\section{Perturbation Thoery}\label{PT}

		Starting with a domain $\Omega_0$ and a degenerate periodic orbit $\gamma$, we introduce a one parameter family of domains $\Omega_\eps$ specified by a smooth function $\mu$:
		\begin{align*}
			\Omega_{\eps} = \{x+ \mu(x, \eps) \nu_x : x \in \d \Omega_0\},
		\end{align*}
		where $\nu_x$ is the outward pointing normal to $\Omega_0$ at $x$. We impose the following constraints on $\Omega_\eps$:
	\begin{enumerate}\label{conditions}
		\item $\Omega_0 = \Omega$.
		\item $\Omega_{\eps}$ and $\Omega_0$ make first order contact at the reflection points $x_i \in \d \gamma$.
		\item Near each $x_i$, $\Omega_{\eps}$ is given by $x + \mu(x, \eps) \nu(x)$, where $x \in \d \Omega$ and $\nu (x)$ is the unit normal at $x$.
		\item $\mu$ is $C^\infty$ in all parameters.
		\item $\det \d^2 \LL \sim c_\gamma \eps$ for $c_\gamma \neq 0$.
		\item $\d^2 \LL$ has rank $q-1$ on $\Omega_0$.
		\item The third and higher order derivatives of $\LL$ are $O(1)$ near $\d \gamma$.
	\end{enumerate}
	In particular, the perturbation preserves periodicity of $\gamma$ and removes its degeneracy. It will be important for us that the Hessian of $\LL$ on $\d \Omega_0$ be minimally degenerate in the sense that it has rank $q-1$. Theorem \ref{main} applies to any family of domains which satisfy the conditions above. To demonstrate the robustness of our main theorem, we now give a large class of examples to which it applies.

	\begin{def1}
		Given a domain $\Omega$ and a $(p,q)$-periodic orbit $\gamma$, we say that $\gamma$ satisfies the injectivity condition if its relfection points are all distinct.
	\end{def1}
	
%

	\begin{prop}\label{q1}
		Let $\mathcal{C}_1$ be the class of smooth, bounded strictly convex planar domains $\Omega$ which have a degenerate periodic orbit $\gamma$ satisfying the injectivity condition. Within $\mathcal{C}_1$, there is an open and dense set in the $C^\infty$ topology on which the Hessian of $\LL$ has rank $q-1$.
	\end{prop}

	\begin{prop}
		Given a domain $\Omega_0$ and a degenerate orbit $\gamma$ having $\text{rank}(\d^2 \LL(S_\gamma)) = q-1$, consider the class $\mathcal{C}_2$ of deformations $\mu: \d \Omega \times [0, \eps_0) \to \R$ which preserve $\gamma$ and the angles of reflection. Within $\mathcal{C}_2$, those which make $\d^2 \LL$ nondegenerate at $S_\gamma$ in arclength coordinates form an open set. If in addition, the orbit satisfies the injectivity condition, then they are also dense in $\mathcal{C}_2$.
	\end{prop}
	
	The preceding two propositions follow from the multijet transversality theorem, which is a generalization due to Mather of Thom's transversality theorem. It is easy to check that for ellipses, the Hessian has rank $q-1$. However, degenerate homoclinic orbits such as those constructed in  \cite{Callis22} could have lower rank Hessians. In \cite{KaloshinZhangRationalCaustics}, it is shown that domains with a rational caustic of rotation number $1/q$ are quantitatively dense in the space of all convex domains. For analytic domains, the density is exponential in $q$, whereas for domains of finite smoothness, the density is polynomial. For such domains, it is easy to see that the length functional evaluated at reflection points of an orbit tangent to a corresponding rational caustic has Hessian of rank $q-1$. Together with the cancellations in our forthcoming paper, we see that the any finite order singular support of the wave trace is generically distinct from the length spectrum in the sense above.
	
	
	\subsection{Nonsingular perturbations}

	We first show that the jet of the length functional is algebraically equivalent to that of the boundary curvature at reflection points and in particular, demonstrate that arbitrarily small perturbations of curvature can result in nondegeneracy. In what follows, write $\det \LL_{\epsilon}^{-1} = M_{\eps}$.
	
	\begin{prop} \label{NSP}
		If the rank of $\d^2 \LL$ is $q-1$, then the condition $\det \d^2 \LL_{\eps} \neq 0$ imposes a codimension $1$ constraint on the $q$-fold $0$-jet of the boundary curvature at reflection points. For each point $x_i \in \d \gamma$ and each $1 \leq i \leq q$, the data $\theta_i$, $|x_i - x_{i+1}|$ and $\d_{i_1, \cdots, i_m}^m \LL$, with $|i_j - i_k| \leq 1$, uniquely determine the $q$-fold $m-2$-jet of the curvature at $x_i \in \d \gamma$. 
	\end{prop}
	\begin{proof}
		It is shown in \cite{KozTresh89} that
		\begin{align}\label{Hessian1}
			\begin{split}
				\delta_i &:=\frac{\d^2\LL}{\d^2 s_i} = 2 \cos \vartheta_i \left( \frac{\cos\vartheta_i}{\ell_i} - \kappa_i\right),\\
				\alpha_i &:=	\frac{\d^2\LL}{\d s_i \d s_{i+1}} =  \frac{\cos \vartheta_i \cos \vartheta_{i + 1}}{\ell_i},
			\end{split}
		\end{align}
		where $\ell_i = |x_i (S) - x_{i+1}(S)|$, $\vartheta_i$ is the angle of incidence with the inward pointing normal at $x_i(S)$, and $\kappa_i$ is the curvature of $\d \Omega$ at $x_i(S)$. It follows that $\d^2 \LL$ is tridiagonal and cyclic with $\delta_i$s on the diagonal and $\alpha_i$s on the off diagonals. The Hessian matrix is of the form
		\begin{align}\label{Hessian2}
			\d^2 \LL = 
			\begin{pmatrix}
				\delta_1 & \alpha_1 & 0 & 0 & \cdots & \alpha_q\\
				\alpha_1 & \delta_2 &\alpha_2  & 0 & \cdots & 0\\
				0 & \alpha_2 & \delta_3 & \alpha_3 & \cdots & 0\\
				0 & 0 & \alpha_3 & \delta_4 & \cdots & 0\\
				\vdots & \vdots & \vdots & \ddots & \vdots & \vdots\\
				\alpha_q & 0 & 0 &  0 & \alpha_{q - 1} & \delta_q
			\end{pmatrix}
		\end{align}
		and the equivalence of curvature and length functional jets follows immediately. Note that the angles of incidence and lengths of links are preserved under deformations which make first order contact with $\d \Omega_0$ at $\d \gamma$. However, the curvature $\kappa_i$ in the formula for $\delta_i$ can change. If the perturbed curvature has an asymptotic expansion of the form
		\begin{align*}
			\kappa_i(\eps) \sim \sum_{j = 0}^{\infty} \mu_j^i \eps^j
		\end{align*}
		for some smooth functions $\mu_j^i(x)$, then we have
		\begin{align*}
			\det \d^2 \LL_{\eps} &= 0 + \eps \Tr \left(\text{adj} \left(\d^2 \LL_0 \right) \d_\eps \d^2 \LL_\eps \big|_{\eps = 0}\right) + O(\eps^2)\\
			& = - \Tr(M \mu) \eps + O(\eps^2),
		\end{align*}
		where $\mu$ is the diagonal matrix given by $\mu_1^i \delta_{ij}$ and $M$ is the adjugate of $\d^2 \LL_0$ consisting of principal minors. For any matrix, its rank is given by the maximal size for which there is a nonzero minor. If $\d^2 \LL$ has rank $q-1$, then the rank of $M$ is $1$. Hence, the condition $\Tr M \mu \neq 0$ imposes a codimension one, linear constraint on the perturbation matrix $\mu$ or equivalently, on the curvature at reflection points.
	\end{proof}
	From the formulas \ref{Hessian1} and \ref{Hessian2} above, one sees that the rank of $\d^2 \LL$ is always at least $q-2$ due to the off diagonal terms being nonzero for nonglancing orbits. If it is in fact equal to $q-2$, then the linearized Poincar\'e map is equal to the identity. If the rank of $\d^2 \LL$ is $q-1$, the linearized Poincar\'e map has $1$ as an eigenvalue but differs from the identity and if the rank is $q$, then the orbit is nondegenerate.

	\subsection{Analysis of the Hessian}
	
	\begin{prop}\label{HA}
		If $\d^2 \LL$ has rank $q-1$ in arclength coordinates at a degenerate periodic trajectory $\gamma$ in $\Omega_0$, then its adjugate matrix has rank $1$ and its elements are of the form $h_{i_1, i_2}(S,0) = \pm h_{i_1}(S,0) h_{i_2}(S,0)$ for some smooth functions $h_i(S,0)$. If a deformation $\Omega_\eps$ which preserves the points and angles of reflection makes $|\det \d^2 \LL| \sim c_\gamma \eps$, $c_\gamma > 0$, then the $(i_1, i_2)$ entry of $\d^2 \LL^{-1}$ is $\wt{h}_{i_1 i_2} = \sigma \cdot \wt{h}_{i_1} \wt{h}_{i_2} + O(1)$ with each $\wt{h}_i$ having the form $(c_\gamma \eps)^{-1/2} h_i$ for some smooth functions $h_i \in C^\infty(\R / \ell \Z \times [0, \eps_0))$ which are uniformly bounded in $\eps$. Here, $\sigma$ is $+1$ if closest eigenvalue to zero of $\d^2 \LL$ is positive and $-1$ otherwise. If $\Omega_0$ is a circle, then $h_{i_1} = h_{i_2}$ for all $1 \leq i_1,i_2 \leq q$.
	\end{prop}
	
	\begin{proof}
		From Proposition \ref{NSP} above, we can choose a perturbation $\mu$ such that $\det \d^2 \LL \sim c_\gamma \eps$, $c_\gamma \neq 0$. Denote by $H = H_{\eps} = \d^2 \LL$ the Hessian of $\LL$ on $\d \Omega_{\eps}$. The adjugate matrix $M_{\eps}$ of $\d^2 \LL$ in $\Omega_{\eps}$ consists of complimentary minors depending on $\|\ka\|_{C^0(\d \Omega_\eps)}$. If the perturbation $\mu(\eps, \cdot)$ makes $\det \LL_{\eps}\sim c_\gamma \eps \neq 0$, then the inverse Hessians satisfy
		\begin{align*}
			\d^2\LL^{-1} \sim \frac{1}{c_\gamma \eps} M_\eps,
		\end{align*}
		where the cofactor matrices are $C^\infty(\R^{q^2} \times [0, \eps_0])$ and uniformly bounded in $\eps$ all the way down to $\eps = 0$. In particular, $\d^2 \LL^{-1} \sim (c_\gamma \eps)^{-1} M_0 + O(1)$. Under the rank $q-1$ assumption, $M_0$ is rank $1$, which together with symmetry about the diagonal implies that $\d^2 \LL^{-1} \sim \sigma \cdot (c_\gamma \eps)^{-1} {h}_{i_1}  {h}_{i_2}$ as in the statement of the theorem.
		\\
		\\
		If $\Omega_0$ is a circle, then we have action angle variables in which the angles are just the usual arclength coordinates on $\d \Omega_0$ and the actions are the radii of concentric circles to which corresponding orbits are tangent. It is easy to see in this case that $\d^2 \LL_0$ is negative semi-definite with rank $q-1$ and has kernel $\R v$, for $v= (1, \cdots, 1)$ corresponding to rotation along a concentric circle. We claim that each row of $M_0$ is a multiple of $v$. Let $M^i$ denote the $i$th row of $M_0$ and choose any $w$ orthogonal to $v$. Then, then using the definition of $M$ in terms of complimentary minors, we have
		\begin{align*}
			\langle M^1, w \rangle = \det
			\begin{pmatrix}
				w_1 & w_2 & \cdots & w_q\\
				H_0^{21} & H_0^{22} & \cdots & H_0^{2q}\\
				\vdots & \hdots & \ddots & \vdots\\
				H_0^{q1} & H_0^{q2} & \hdots & H_0^{qq}
			\end{pmatrix},
		\end{align*}
		where $H_0^{ij}$ are the entries of $\d^2 \LL_0$. Since $w \perp \ker H_0 = \ker H_0^*$, $w \in \Span\{H_0^1, H_0^2, \cdots, H_0^q\} = \Span \{H_0^2, \cdots, H_0^q\}$, given that the rows sum to zero and are hence linearly dependent. Therefore, $\langle M^1, w \rangle = 0$. We have shown that $\Span(v)^\perp \subset \Span(M^1)^\perp$ which implies $\Span(M^1) = \Span(v)$ since they have equal dimensions; the same holds if one replaces $1$ by any $1 \leq i \leq q$. Each row of $M$ is a multiple of $v$ and by symmetry about the diagonal, the multiples must be the same. In particular,
		\begin{align*}
			\d^2 \LL^{-1} \sim \frac{\sigma}{c_\gamma \eps} h (1, \cdots, 1)^T (1, \cdots, 1) + O(1),
		\end{align*}
		for a single smooth function $h$ on $\Omega_0$. Since $\d^2 \LL$ has $q-1$ negative eigenvalues on the disk, $\sigma$ only depends on ellipticity or hyperbolicity of the corresponding perturbed orbit.
%
%
%
	\end{proof}

	\section{Stationary Phase and Feynman Diagrams}\label{stationary phase and feynman diagrams}
	We now apply the method of stationary phase to the integral appearing in Theorem \ref{SP1}, but on the perturbed domain $\Omega_{\eps}$ from Section \ref{PT}, so as to make the phase nondegenerate. In general, if $\Phi$ has a unique stationary point at $x_0$ with nondegenerate Hessian and $u$ is a compactly supported smooth function, the stationary phase formula in \cite{Ho90} reads:
	\begin{align}\label{Zk}
		\int_{\R^n} u(x) e^{i k \Phi(x)} dx \sim (2\pi/k)^{n/2}  \frac{e^{ik \Phi(x_0)} e^{i \pi \text{sgn} \d^2 \Phi(x_0)/4}}{ |\det \d^2 \Phi (x_0)|^{1/2} } \sum_{j = 0}^\infty k^{-j} L_j u(x_0),
	\end{align}
	where the $L_j$ are differential operators of order $2j$ having the form
	\begin{align}\label{Zk2}
		L_j u (x_0) = \sum_{\nu - \mu = j} \sum_{ 2\nu \geq 3 \mu} i^{-j} 2^{-\nu}  \langle \d^2 \Phi(x_0)^{-1} \d , \d \rangle^\nu (g^\mu u (x_0))/\mu! \nu!.
	\end{align}
	Here, $g = \Phi(x) - \Phi(x_0) - \Phi'(x_0) (x- x_0) - \Phi''(x_0) (x-x_0)^2$ is the higher order part of the Taylor expansion of $\Phi$, having polynomial order $\geq 3$. We in fact need a version of stationary phase which allows for dependence on auxilary parameters. If the phase $\Phi$ and amplitude $u$ depend smoothly on a parameter $\tau \in T$ with $u$ being supported in a compact set which is independent of $\tau$, then expansion \ref{Zk} is uniform in $\tau$ (see \cite{Du96}). There are several asymptotics to keep track of simultaneously: powers of $k$ and $\eps$ together with the number of boundary curvature derivatives.
	\\
	\\
	We use Feynman diagrams to organize the terms hierarchically, inspired by but using a different structure than those in \cite{Zel09}.
	
	\subsection{Feynman Calculus}
	We review the procedure first used by Feynman to evaluate asymptotic integrals. We will follow the conventions on notation as in \cite{axelrod}.
	\begin{def1}
		A graph $\mathcal{G}$ with $V$ vertices and $I$ edges is called a Feynman diagram of order $j$ if
		\begin{enumerate}
			\item $I - V = j$,
			\item there are $V$ closed (indistinguishable) vertices, each of which has valency at least $3$,
			\item there is one open, marked vertex with any possible valency, including $0$ (corresponding to an isolated vertex), and
			\item each edge is labeled by a function $\ell$, which assigns two indices $(i_1,i_2)$ corresponding two its two endpoints.
		\end{enumerate}
	\end{def1}
\begin{rema}
	The quantity $I - V = j$ is closely related to the Euler characteristic $\chi(\mathcal{G})$, but does not include faces; not all diagrams are assumed to be planar graphs. We will sometimes denote $\mathcal{G}$ by $\mathcal{G}_\ell$ to emphasize $\ell$, the edge labeling.
\end{rema}
	\begin{def1}
		Let $\mathcal{G}$ be a Feynman diagram of order $j$ with labeling $\ell$.
		\begin{itemize}
			\item For the open vertex, we associate the quantity
			$$\frac{\d^v u}{\d x_{i_1} \cdots \d x_{i_v}} (x_0),
			$$
			where $i_1, \cdots, i_v$ are the indices labeling the incident edges ($v$ is the valency of the vertex).
			\item To each closed vertex corresponds a factor of
			$$
			ik \frac{\d^v \Phi}{\d x_{i_1} \cdots \d x_{i_v}}(x_0),
			$$
			where again $i_1, \cdots, i_v$ are the indices of the incident edges.
			\item For each edge with indices $i_1, i_2$, we assign the inverse Hessian element
			$$
			\frac{i}{k} (\d^2 \Phi)_{i_1 i_2}^{-1} := \frac{i}{k} \wt{h}_{i_1 i_2}.
			$$
		\end{itemize}
	The Feynman amplitude $F_{\mathcal{G}, \ell}$ corresponding to $\mathcal{G}$ and label $\ell$ is given by the product of each factor above.
	\end{def1}
	
	\begin{prop}[\cite{axelrod}]\label{Feynman} 
		Let $Z_k$ denote the integral \ref{Zk}. As $k \to \infty$,
		\begin{align}
			Z_k \sim \left( \frac{2\pi}{k}\right)^{n/2} \frac{e^{i k \Phi(x_0)} e^{i \pi \text{sgn} \d^2 \Phi }}{|\det \d^2 \Phi|^{1/2}} \sum_{ I = 0}^\infty \sum_{V = 0}^\infty \sum_{\mathcal{G}, \ell} \frac{F_\mathcal{G}}{w(\mathcal{G})},
		\end{align}
		where the inner sum is over all labeled labeled Feynman diagrams with $V$ internal vertices, $1$ external vertex, and $I$ edges. $\omega(\mathcal{G})$ is the order of the automorphism group of $\mathcal{G}$, $F_{\mathcal{G}, \ell}$ is the Feynman amplitude and $\ell$ assigns to each endpoint of each edge an index in $\{1, \cdots, n\}$.
	\end{prop}
	We refer to \cite{axelrod} for an elementary proof of the equivalence of \ref{Zk} and Proposition \ref{Feynman}. It remains evaluate each of these Feynman amplitudes in coordinates to determine their contributions to $B_{\gamma, j}$.

\subsection{Contributing Graphs}
Here, we examine the Feynman diagrams which contribute maximal inverse powers of $\eps$ to the trace expansion in Proposition \ref{SP1}. It is clear that the amplitudes $a_0^\sigma(k+i \tau, S ; \eps)$ are uniformly bounded in $\epsilon$. The main contributions come from the closed vertices (derivatives of the phase) and their incident edges (inverse Hessians). To keep track of leading order terms in $\epsilon$, we must isolate those Feynman diagrams which produce the most inverse Hessians, each of which produces a power of $\eps^{-1}$.

\begin{lemm}\label{F3}
	For each $j$, the Feynman amplitudes having maximal products of inverse Hessian elements correspond to the class of $3$-regular graphs on $2j$ closed vertices with one additional open vertex having valency zero.
\end{lemm}
\begin{proof}
	In the usual non-diagrammatic stationary phase notation, maximal inverse Hessians arise from the terms where $\nu - \mu = j$ and $2 \nu = 3\mu$, which yields $\nu = 2j$ and $\mu = 3j$. In Proposition \ref{Feynman}, $V - I = j$ and $\mathcal{G}$ has $I = 3j$ edges. For a general graph,
	 $$
	 I = \frac{u_0}{2} + \frac{1}{2} \sum_{i = 1}^V v_i,
	 $$
	 where $u_0$ is the valency of the open vertex and $v_i$ are the valencies of unmarked vertices. As each $v_i \geq 3$ by assumption, we have
	 \begin{align*}
	 	I - \frac{3}{2} V\geq& 0,\\
	 	- \frac{3}{2} I + \frac{3}{2} V =& - \frac{3}{2} j,
	 \end{align*}
	 from which it follows that $I \leq 3j$, which is the maximal number of inverse Hessians. This is acheived when $V = 2j$ and $I = 3j$, corresponding to each $v_i = 3$. If the open vertex were not isolated, consider the graph $\mathring{\mathcal{G}} = \mathcal{G} \backslash \{\text{open vertex and all its incident edges}\}$. The formula above would then force there to be vertices with valency strictly less than $3$.
\end{proof}

\begin{figure}
	\includegraphics[scale = 0.35]{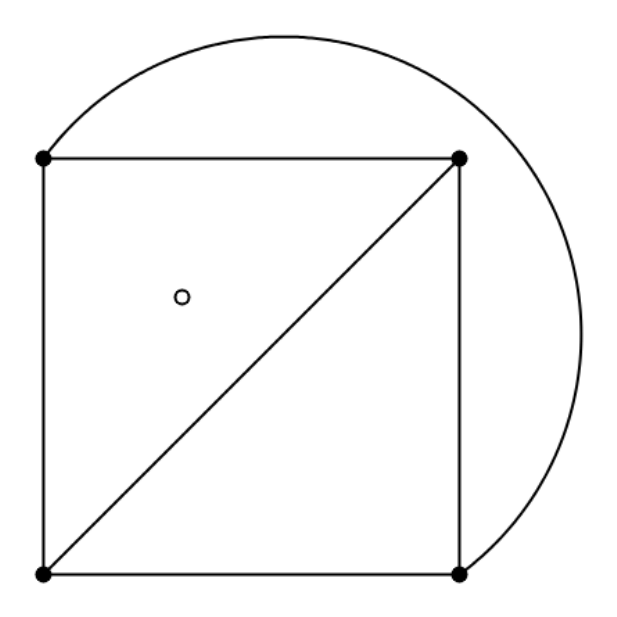}
	\includegraphics[scale = 0.35]{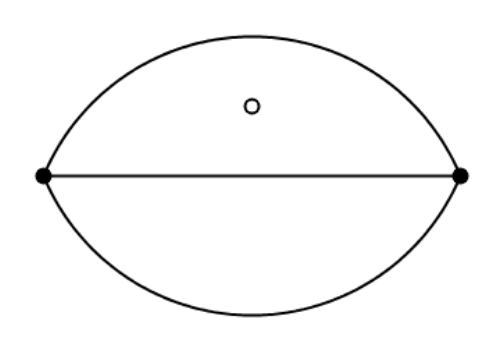}
	\includegraphics[scale = 0.35]{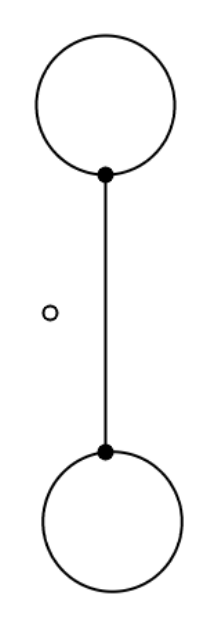}\\
	\includegraphics[scale = 0.35]{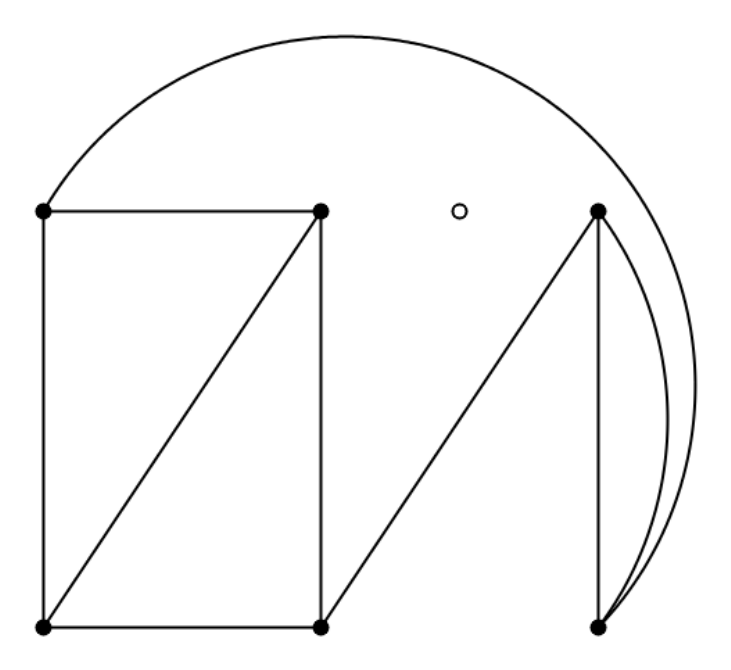}
	\includegraphics[scale = 0.35]{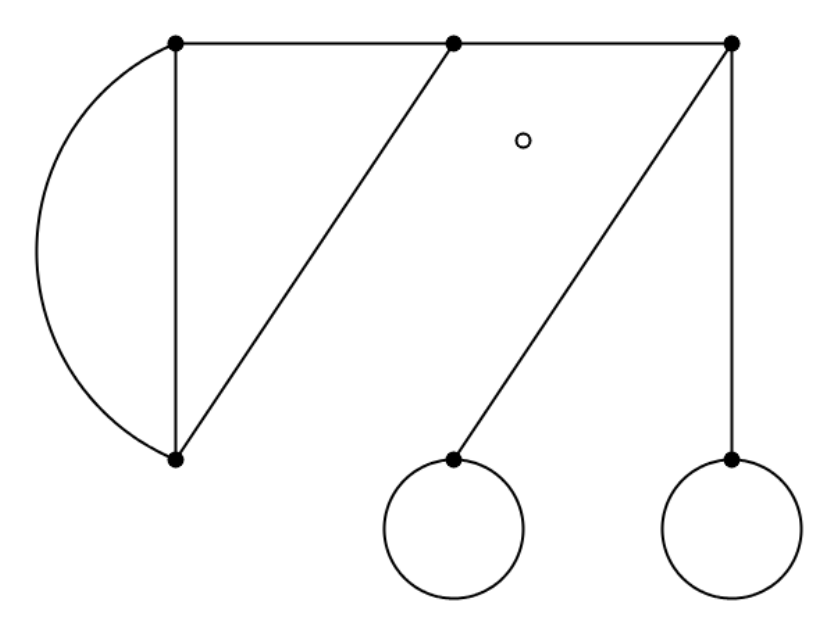}
	\label{FeynDiag}
	\caption{Some admissible Feynman diagrams given by $3$ regular graphs on $2j$ vertices.}
\end{figure}

\begin{coro}\label{FAmp} Modulo an error bounded uniformly in $\eps$, the Feynman amplitudes for contributing graphs with maximal Hessians consist of $a_0(S_\gamma) (\pm i)^j k^{-j}$ and products of terms of the form
	\begin{align*}
		(c_\gamma \eps)^{-\frac{3}{2}} h_{i_1} h_{i_2} h_{i_3} \sum_{i_1, i_2, i_3} \d_{i_1 i_2 i_3}^3 \LL,
	\end{align*}
for some indices $1 \leq i_1, i_2, i_3 \leq q$ and $a_0$ the amplitude defined in Proposition \ref{SP1}.
\end{coro}

\subsection{Critical points on the diagonal}

The integral in Theorem \ref{SP1} has two types of critical points. There are those on the large diagonals $x_i = x_{i + 1}$ and regular critical points which are off diagonal.

\begin{lemm}\label{lower order} [\cite{Zel09}]
	The amplitude $A_\sigma$ of $\n_\sigma$ is a semiclassical amplitude of order $-|\sigma|$.
\end{lemm}
Hence, the critical points on the large diagonal have lower order and we can focus on the portion of $N$ which quantizes the $M$-fold billiard map on $S_+^* \d \Omega$.

\begin{coro}\label{LO}
	The coefficient of $k^{j}$ in the asymptotic expansion of the integral coming from $\n_\sigma$ in Theorem \ref{SP1} is $O(\epsilon^{-3j + |\sigma|})$.
\end{coro}
\begin{proof}
	By Lemma \ref{lower order}, $\Tr \rho' \ast \n_\sigma(k)$ has an asymptotic expansion in negative powers of $k$ begining with $k^{- |\sigma|}$. Hence, the coefficient of $k^{-j}$ corresponds to the $j - |\sigma|$ term in the usual stationary phase expansion \ref{Zk} and the maximal contribution of inverse Hessians to such a term is $3( j - |\sigma|) \leq 3j$.
\end{proof}

When $|\sigma| > 0$, these terms are lower order in $1/\epsilon$ and can be included in the remainder so it suffices to only consider the contributions of regular critical points in the limit $\epsilon \to 0$. Notice also that the $A_\sigma'$ term in the amplitude $a_0$ appearing in Theorem \ref{SP1} is a semiclassical symbol of order $-1 - |\sigma|$ and can be discarded from the leading order asymptotics in $k$ by the same reasoning as in Corollary \ref{LO}.


\subsection{Proof of Theorem \ref{main}}
We can now complete the derivation of the regularized resolvent trace formula appearing in Theorem \ref{main}.

\begin{proof}
	We will analyze the expression
		\begin{align*}
		\mathcal{D}_{\gamma}(\lambda)   \sum_{j = 0}^\infty B_{\gamma, j} \Re(\lambda)^{-j} \sim^* \frac{1}{2} \sum_{M = 1}^\infty \sum_{\sigma: \Z_M \to \{0,1\}} \frac{e^{ -M \pi i / 4}}{M} \int_{{\d \Omega}_S^{M}}  e^{i \lambda \LL(S)} a_0^\sigma(\lambda, S) dS
	\end{align*}
	in Proposition \ref{SP1} by first expanding the amplitude $a_0^\sigma$ in inverse powers of $k$ from the representation \ref{Hankel} and then replacing each term by its stationary phase expansion. Recall that
	\begin{align*}
		a_0^\sigma (k +i \tau, S) = \hat{\rho}(\LL(S)) \left(\LL(S)  A_\sigma (k + i \tau, S) - \frac{1}{i}  \frac{\d}{\d k} A_\sigma (k + i \tau, S)  \right),
	\end{align*}
	with $A_\sigma ( k + i \tau, S)$ being
	\begin{align*}
		( k + i \tau)^{M/2}  \prod_{i = 1}^{M}  \frac{\chi_i^\sigma (k + i \tau,S)}{ |x(s_i) - x(s_{i+1})|^{1/2}}  a_1(( k + i \tau) |x(s_i) - x(s_{i+1})|) \cos \theta_i,
	\end{align*}
	and $a_1$ the Hankel amplitude which has the asymptotic expansion \ref{Hankel}. It is shown in \cite{Zel09} (Corollary 3.9 and the proof of Proposition 3.10, together with Lemma 5.8) that
	\begin{align*}
		\int_{\d \Omega^M} e^{i (k + i \tau) \LL(S)} a_0^\sigma ((k+i \tau), S) dS = \int_{\d \Omega^{M - |\sigma|}} e^{i (k + i\tau) \LL(S)} b_0^\sigma ((k+i \tau), S) dS,
	\end{align*}
	for a symbol $b_0^\sigma \in S_\varrho^{-|\sigma|}(\d \Omega^{M - |\sigma|})$ having an asymptotic expansion of the form
	\begin{align*}
		b_0^\sigma(S, k + i \tau) \sim \sum_p b_{0,p}^\sigma(S) k^{- \sigma - p},
	\end{align*}
	with each $b_{0,p}^\sigma$ depending on at most $p$ derivatives of the boundary curvature.  From this, an application of the stationary phase lemma gives
	\begin{align*}
		\int_{\d \Omega^M} e^{i k \LL(S)} a_0^\sigma ((k+i \tau), S) dS \sim \left(\frac{2\pi}{k}\right)^{\frac{M - |\sigma|}{2}} \frac{e^{i k \LL(S_\gamma)} e^{i\pi \sgn \d^2 \LL(S_\gamma)/4}}{\sqrt{|\det \d^2 \LL(S_\gamma)|}} \sum_{j} c_{j, \sigma} k^{-j - |\sigma|}
	\end{align*}
	where $c_{j, \sigma}$ has $3j$ inverse Hessians. Having localized $\hat{\rho}$ near $L_\gamma$, the integral is $O(k^{-\infty})$ unless $M - |\sigma| = q$, in which case the coefficient of any $k^{-j_0}$ in the expansion has $3(j_0 - |\sigma|)$ inverse Hessians. After separating out the prefactor and using the fact that the derivatives of angles at reflection points depend only on the first order curvature jet, we see that the nonzero $\sigma$ terms contribute $3j - 3|\sigma|$ inverse Hessians multiplied by a coefficient which depends smoothly on parameters and a priori on the first $6(j - |\sigma|)$ derivatives of the boundary curvature. However, it is shown in \cite{Zel09} that this can be reduced to $2j$ derivatives of boundary curvature, essentially because in \ref{Zk} and \ref{Zk2}, powers of the higher order phase expansion $g^\mu$ vanish to order $3\mu$. It takes $3\mu$ derivatives to remove this zero, which leaves at most $2j - \mu$ derivatives left to act on the phase or amplitude. The amplitude only depends on the $0$-jet of $\ka$ and $g$ depends on the $1$-jet, but only appears when $\mu \neq 0$. In either case, at most $2j$ derivatives of $\ka$ emerge.
	\\
	\\
	We now restrict our attention to the case $|\sigma| = 0$. Since $|x(s_i) - x(s_{i+1})| \leq \text{diam}(\Omega) \leq |\d \Omega|$, we have that for $S$ near $S_\gamma$,
	\begin{align*}
		a_0(k + i \tau, S) = b_0^0(S, k + i \tau) = \left(\frac{k + i \tau}{2 \pi}\right)^{\frac{q}{2}} \LL(S) \prod_{i = 1}^q \frac{\cos \theta_i }{|x(s_i) - x(s_{i+1})|} + \mathcal{R}_{a_0, 0}(k + i \tau, S),
	\end{align*}
	where the remainder satisfies
	\begin{align*}
		|\d_z^\alpha \d_k^\beta \mathcal{R}_{a_0, 0}(k + i \tau, S)| \leq C_{\alpha, \beta, \varrho, |\d \Omega|, q} \langle k \rangle^{\varrho |\alpha| - |\beta| - \frac{q}{2} - 1}
	\end{align*}
	uniformly and with smooth dependence on parameters, including $\eps$ and the full jet of $\ka$. Here, we have chosen only leading order terms in the asymptotic expansion of $a_1$. The contributions of $\mathcal{R}_{a_0, 0}$ and $\frac{\d}{\d k} A_\sigma \in S_\varrho^{-1}$ to the $k^{-j}$ term in the stationary phase expansion have submaximal inverse Hessians for essentially the same reason as the $|\sigma| \geq 1$ terms: an asymptotic expansion of either which starts with $k^{-\wt{j}}$ ($\wt{j} \geq 1$) contributes to the $k^{-j}$ term in stationary phase expansion through the operator $L_{j - \wt{j}}$, which has at most $3(j - \wt{j})$ inverse Hessians.
	\\
	\\
	Thus, for the purposes of calculating wave invariants modulo lower order inverse Hessians, we may replace the integral in Proposition \ref{SP1} by
	\begin{align}\label{HO}
		\left(\frac{k + i \tau}{2\pi}\right)^{\frac{q}{2}} \int_{\d \Omega^q} e^{i (k + i\tau) \LL(S)} \LL(S)  \prod_{i = 1}^q \frac{\cos \theta_i }{|x(s_i) - x(s_{i+1})|} dS,
	\end{align}
		to which we apply the stationary phase lemma with parameter $\tau \in [0,1]$. As noted in Section \ref{Balian-Bloch Invariants Section}, holomorphy of the regularized resolvent trace allows us to take $\tau \to 0^+$. Denote by $\wt{a}(\lambda, S)$ the integrand of \ref{HO}, corresponding to the amplitude with $|\sigma| = 0$ and highest powers of $\eps^{-1}$. To find explicitly the stationary phase coefficients, we apply the Feynman rules. In Corollary \ref{FAmp}, we saw the terms which appear in each Feynman amplitide. Most labeled diagrams contribute a Feynman amplitude of {zero}, since the derivatives of $\LL$ are only nonzero if $\max\{|i_1 - i_2|, |i_1 - i_3|, |i_2-i_3|\} \leq 1$. Hence, the combinatorial constant depends only on the order of the automorphism group of $3$-regular graphs on $2j$ vertices.
	Plugging in $\wt{a}$ and $\LL$ to the formula in Proposition \ref{Feynman}, and considering a $3$-regular graph $\mathcal{G}$, we sum over all labelings to obtain the Feynman amplitude
	\begin{align*}
		\sum_{\ell} F_{\mathcal{G},\ell} =& \sum_{\ell} \wt{a}(S_\gamma) \prod_{\text{edges}} \prod_{\text{vertices}} \left(i k \d_{i_1(\ell), i_2(\ell), i_3(\ell)} ^3 \LL \right) \left(\frac{i}{k} \wt{h}_{j_1(\ell) j_2(\ell)}\right)\\
		= &  (\pm i)^j k^{-j} \left(c_\gamma \eps\right)^{-3j} \wt{a}(S_\gamma)  \left(\sum_{i_1, i_2, i_3} h_{i_1} h_{i_2} h_{i_3} \d_{i_1, i_2, i_3}^3 \LL(S_\gamma) \right)^{2j}.
	\end{align*}
	Putting $M = q$, $|\sigma| = 0$, dividing by the order of the automorphism groups, summing over all Feynman diagrams and collecting remainder terms finishes the proof.
\end{proof}

\section{Future directions}
As mentioned in the introduction, this paper is part of a series in which we aim to show that the singular support of the wave trace and the length spectrum are inherently different objects. Our subsequent paper will produce cancellations to arbitrarily high orders in the wave trace. It would also be interesting to both derive degenerate asymptotics and create cancellations in:
\begin{itemize}
	\item the wave trace for Riemannian manifolds without boundary,
	\item the semiclassical trace formula for the density of states of a Schr\"odinger operator on a Riemannian manifold with or without boundary, and
	\item and the semiclassical trace formula for eigenvalues of large graphs.
\end{itemize}
In the setting of convex planar domains considered in the present paper, the billiard ball map forms a sort of global Poincar\'e expression for the billiard \textit{flow}, which consists of broken bicharacteristics. Formally, our result can then be interpreted as an asymptotic expansion of the trace of a certain boundary operator which quantizes the Poincar\'e map, analogous to the approach in \cite{SZ} and \cite{ISZ2}. The authors there use an abstract reduction to a Grushin problem and in future work, it would be interesting to compare their method with the layer potential approach we employ here.

\section{Acknowledgements}
The first and second authors acknowledge the partial support of the ERC Grant \#885707. The third author is grateful to IST Austria and the Schr\"odinger Institute in Vienna for hosting him during much of the writing of this paper. The authors would also like to thank Hamid Hezari for suggesting the Balian-Bloch approach and the late Steven Morris Zelditch for many helpful conversations in the begining of this project. We are grateful to have known Zelditch for many years and are deeply saddened by his passing. His contributions have made a profound impact on many fields in mathematics. His energy, passion and enthusiasm uplifted all those around him. He was a great inspiration to many mathematicians, including ourselves.

	\bibliographystyle{alpha}
	\bibliography{MB1}

\end{document}